%% file: cmethod.tex
\documentclass[reqno,oneside,letterpaper]{amsart}
\usepackage{mathrsfs}
\input{headers.tex}

\usepackage{setspace}
\usepackage{bbm}
\input{defs.tex} 

\makeatletter
\providecommand*{\toclevel@definition}{0}
\providecommand*{\toclevel@theorem}{0}
\providecommand*{\toclevel@lemma}{0}
\makeatother

\usepackage{leftidx}
\title[Nonstandard Representation of the Dirichlet Form]
{
Nonstandard Representation of the Dirichlet Form
}

\newcommand{\cB}{\mathscr{B}}

\newcommand{\topology}{\mathcal{T}}

\newcommand{\Diric}[3]{\mathscr{E}_{#1}(#2,#3)}
\newcommand{\hyDiric}[3]{\mathscr{F}_{#1}(#2,#3)}

\newcommand{\cR}{\mathscr{R}}
\newcommand{\cP}{\mathscr{P}}
\newcommand{\indicate}{\mathbbm{1}}
\newcommand{\hycongest}{\mathbb{B}}
\newcommand{\alt}[1]{\tilde{#1}}
\newcommand{\PathSpace}{\digamma}
\newcommand{\hyper}[1]{\mathcal{#1}}
\newcommand{\HPathSpace}{\Upsilon}
\newcommand{\goodset}{\mathcal{C}}
\newcommand{\biggoodset}{\mathcal{O}}
\newcommand{\proj}{\mathbbm{P}}
\newcommand{\allpath}{\Theta}

\newcommand{\Diagonal}{\mathbb{D}}

\newcommand{\metric}{W}

\begin{document}

\author[R.~M.~Anderson]{Robert M.~Anderson}
\address{University of California, Berkeley, Department of Economics}

\author[H.~Duanmu]{Haosui Duanmu}
\address{University of California, Berkeley, Department of Economics}

\author[A.~Smith]{Aaron Smith}
\address{University of Ottawa, Department of Mathematics and Statistics}

\maketitle

\begin{abstract}
\input{compabstract.tex}
\end{abstract}

%

\input{cmethodpub.tex}

%


\end{document}

%% file: headers.tex
\usepackage[utf8]{inputenc}
\usepackage{amsmath, amssymb,bm, cases, mathtools, thmtools}
\usepackage{verbatim}
\usepackage{graphicx}\graphicspath{{figures/}}
\usepackage{multicol}
\usepackage{tabularx}
\usepackage[usenames,dvipsnames]{xcolor}
\usepackage{mathrsfs}
\usepackage{url}
\usepackage[normalem]{ulem}
\usepackage{xstring}
\usepackage{enumitem}

\input{bib.tex}

\usepackage[colorlinks,citecolor=blue,urlcolor=blue,linkcolor=RawSienna]{hyperref}
\usepackage{hypernat}
\usepackage{datetime}

\DeclareMathAlphabet\EuRoman{U}{eur}{m}{n}
\SetMathAlphabet\EuRoman{bold}{U}{eur}{b}{n}

\usepackage{euscript,microtype}

\usepackage[capitalize]{cleveref}

\crefname{assumption}{Assumption}{Assumptions}
\crefname{claim}{Claim}{Claims}

\makeatletter
\let\reftagform@=\tagform@
\def\tagform@#1{\maketag@@@{\ignorespaces\textcolor{gray}{(#1)}\unskip\@@italiccorr}}
\renewcommand{\eqref}[1]{\textup{\reftagform@{\ref{#1}}}}
\makeatother

\input{commenting.tex}

\declaretheorem[style=plain,numberwithin=section,name=Theorem]{theorem}
\declaretheorem[style=plain,sibling=theorem,name=Lemma]{lemma}

\declaretheorem[style=plain,sibling=theorem,name=Claim]{claim}

\declaretheorem[style=definition,sibling=theorem,name=Definition]{definition}
\declaretheorem[style=definition,name=Assumption]{assumption}

\declaretheorem[style=definition,name=Condition]{conditionINNER}
\newenvironment{condition}[1]
 {\renewcommand{\theconditionINNER}{#1}\conditionINNER}
 {\endconditionINNER}
\crefformat{conditionINNER}{#2(#1)#3}
\crefmultiformat{conditionINNER}
  {(#2#1#3)}
  { and~(#2#1#3)}
  {, (#2#1#3)}
  { and~(#2#1#3)}
\Crefformat{conditionINNER}{Condition~#2(#1)#3}
\Crefmultiformat{conditionINNER}
  {Conditions~(#2#1#3)}
  { and~(#2#1#3)}
  {, (#2#1#3)}
  { and~(#2#1#3)}

\declaretheoremstyle[
    spaceabove=-6pt,
    spacebelow=6pt,
    headfont=\normalfont\bfseries,
    bodyfont = \normalfont,
    postheadspace=1em,
    qed=$\square$,
    headpunct={{}}]{myproofstyle}

\numberwithin{equation}{section}
\numberwithin{theorem}{section}

\renewcommand{\appendixname}{}

\usepackage{amssymb}
\usepackage{mathrsfs}
\usepackage{leftidx}

%% file: bib.tex
\usepackage[%
    minnames=1,maxnames=99,maxcitenames=3,
    style=alphabetic,
    doi=false,url=false,
    firstinits=true,hyperref,natbib,backend=bibtex]{biblatex}
\renewbibmacro{in:}{%
  \ifentrytype{article}{}{\printtext{\bibstring{in}\intitlepunct}}}
\bibliography{biblio}

%% file: commenting.tex
\definecolor{WowColor}{rgb}{.75,0,.75}
\definecolor{SubtleColor}{rgb}{0,0,.50}

\newcounter{margincounter}

%% file: defs.tex
\def\[#1\]{\begin{align}#1\end{align}}
\def\*[#1\]{\begin{align*}#1\end{align*}}

\newcommand{\Reals}{\mathbb{R}}
\newcommand{\Nats}{\mathbb{N}}

\newcommand{\PosReals}{\Reals_{> 0}}

\newcommand{\dee}{\mathrm{d}}

\DeclareMathOperator*{\newlim}{\mathrm{lim}\vphantom{\mathrm{infsup}}}
\DeclareMathOperator*{\newmin}{\mathrm{min}\vphantom{\mathrm{infsup}}}
\DeclareMathOperator*{\newmax}{\mathrm{max}\vphantom{\mathrm{infsup}}}
\DeclareMathOperator*{\newinf}{\mathrm{inf}\vphantom{\mathrm{infsup}}}
\DeclareMathOperator*{\newsup}{\mathrm{sup}\vphantom{\mathrm{infsup}}}
\renewcommand{\lim}{\newlim}
\renewcommand{\min}{\newmin}
\renewcommand{\max}{\newmax}
\renewcommand{\inf}{\newinf}
\renewcommand{\sup}{\newsup}

\newcommand{\ceiling}[1]{\lceil #1 \rceil}

\newcommand{\cF}{\mathcal F}

\newcommand{\BorelSets}[1]{\mathcal{B}[#1]}
\newcommand{\NSE}[1]{{^{*}#1}}
\newcommand{\ST}{\mathsf{st}}

\newcommand{\HReals}{\NSE{\Reals}}

\newcommand{\NS}[1]{\mathrm{NS}(#1)}

\newcommand{\cA}{\mathcal{A}}

\newcommand{\cC}{\mathcal{C}}

\newtheorem{open problem}{Open Problem}

\newcommand{\Loeb}[1]{\overline{#1}}

\newcommand{\interior}[1]{%
  {\kern0pt#1}^{\mathrm{o}}%
}

\newcommand{\refproof}[1]{See \cref{#1} for \IfSubStr{#1}{,}{proofs}{a proof}. }

\newif\iflongform
\longformtrue

\iflongform

\else

\fi

%% file: compabstract.tex
The Dirichlet form is a generalization of the Laplacian, heavily used in the study of many diffusion-like processes. In this paper we present a \textit{nonstandard representation theorem} for the Dirichlet form, showing that the usual Dirichlet form can be well-approximated by a \textit{hyperfinite} sum. One of the main motivations for such a result is to provide a tool for directly translating results about Dirichlet forms on finite or countable state spaces to results on more general state spaces, without having to translate the details of the proofs. As an application, we prove a generalization of a well-known comparison theorem for Markov chains on finite state spaces, and also relate our results to previous generalization attempts. 

%% file: cmethodpub.tex
\singlespace

\section{Introduction}

The Dirichlet form, introduced in \citet{Beurling58}, can be used in place of the usual Laplacian in situations where the ``usual" Laplacian may not be convenient or make sense at all. Among other applications, it is an important tool for defining diffusive processes on various complicated spaces (see \textit{e.g.} an early paper on fractals \citet{kusuoka1989dirichlet} and an introductory paper on infinite-dimensional processes \citet{schmuland1999dirichlet}), potential theory (see \textit{e.g.} \citet{albeverio2003lectures} for the application of Dirichlet forms and \citet{doob2012classical} for the classical theory) and for comparing processes (see versions of such results in \citet{davies_1989,LPW09}). See \textit{e.g.} \citet{ma2012introduction,fukushima1996dirichlet} for broad introductions to Dirichlet forms and their uses.

The main result of this paper, \cref{hydstd}, develops a \textit{nonstandard representation theorem} for Dirichlet forms. We give a very informal summary here. Recall that the \textit{Dirichlet form} of a transition kernel $g$ with stationary measure $\pi$ on state space $X$ can be applied to functions $f \in L^{2}(\pi)$ via the formula:

\[
\Diric{g}{f}{f}=\frac{1}{2}\int_{x\in X}\int_{y\in X}[f(x)-f(y)]^{2}g(x,1,\dee y)\pi(\dee x).
\]

 \cref{hydstd} says that, under appropriate conditions, this Dirichlet form can be well-approximated by an appropriate \textit{hyperfinite sum}. With notation to be fixed later in the paper, the main conclusion of the theorem is written

\[ \label{IneqHyperfiniteIntro}
\Diric{g}{f}{f}\approx \frac{1}{2}\sum_{s,t\in S_{X}}[F(s)-F(t)]^{2}H_{s}(\{t\})\Pi(\{s\}),
\]

where $S_{X}$ is a set meant to approximate $X$, $F$ is a function meant to approximate $f$,  $H_{s}$ and $\pi$ are measures meant to approximate $g(s,1,\cdot)$ and $\pi$ respectively, and the symbol ``$\approx$", which is defined precisely in \cref{SecIntroNSA}, means two quantities equal up to an infinitesimal.

The immediate motivation for a result such as \eqref{IneqHyperfiniteIntro} is that the \textit{hyperfinite} sum over $S_{X}$ behaves a great deal like a more-familiar finite sum. This allows one to \textit{directly} translate certain results about processes on discrete spaces to results about processes on more general spaces, without worrying about translating each step in the associated proof. Such direct translation can lead to substantially simpler and shorter proofs, and sometimes allows one to avoid assumptions, \textit{e.g.} by allowing one to bypass differentiability assumptions that would be needed in translating the \textit{proof steps} but which are not needed to translate the \textit{theorem statement.}

As an illustration of \cref{hydstd} and how it may be applied, we give a simple translation of the well-known comparison theorem for Markov chains, first proved in \citet{diaconis1993comparison}. Previous generalizations of this result have been obtained by purely standard results (see \citet{yuen2000applications}). However, our version offers several important improvements - for example, we remove various differentiability conditions (which in fact often fail in practice) and are able to obtain substantially sharper estimates in some basic but important test cases. See Section \ref{sec1dim} for a brief application and discussion, and our forthcoming companion paper focused on these applications for more details.

\subsection{Nonstandard Analysis and Transferring Program}

We view the main contribution of this paper as the development of a nonstandard version of the classical Dirichlet form. This paper is a small part of an ongoing effort to provide nonstandard analogues to a variety of important objects in probability theory and statistics: \citep{Markovpaper}, \citep{anderson2018mixhit}, \citep{anderson2019mixavg}, \citep{anderson2019gibbs}, \citep{nscredible} and \citep{nsbayes}. The main motivation is the rough observation that many interesting theorems in probability have the following properties:

\begin{itemize}
\item The theorem is initially proved on a discrete (or finite-dimensional) space, \textit{and}
\item The theorem \textit{statement} does not seem to rely heavily on the space being discrete or finite-dimensional, \textit{but}
\item Several steps in the most natural proof \textit{do} seem to rely heavily on the space being discrete or finite-dimensional.
\end{itemize}

For some examples, see \textit{e.g.} \citep{andersonisrael}, \citep{Keisler87}, \citep{Markovpaper}, \citep{anderson2018mixhit}, \citep{nscredible} and \citep{nsbayes}. For results that have this form, it is natural to try to directly translate the \textit{theorem statement} without needing to go through the details of translating \textit{the full proof.} This idea of translation is at the heart of nonstandard analysis, where it is formalized in the notion of ``transfer." This basic idea has let us and others make progress on several problems that otherwise appear difficult, including: 
\begin{itemize}
\item In \citet{Markovpaper}, we prove the Markov chain ergodic theorem for a large class of continuous time general Markov processes, generalizing the well-known Markov chain ergodic theorem for discrete Markov processes. 
    
\item In \citet{nsbayes},  we show that a decision procedure is extended admissible if and only if it has infinitesimal excess Bayes risk under a nonstandard prior distribution. This result holds under complete generality and is a generalization of existing complete class theorems.

\item In \citet{anderson2018mixhit}, we show that mixing time and hitting time are asymptotically equivalent for general Markov processes under moderate regularity condition, generalizing the same result for finite Markov processes in \citet{finitemixhit}. 
    
\item In \citet{nscredible}, we show that matching priors for specific families of credible sets exist on compact metric spaces, extending a result for finite spaces in \citet{muller2016coverage}.    
\end{itemize}

As discussed in the introduction, this paper includes another example of the power of this approach: we provide a translation of {\citet[][Thm.~13.23]{markovmix}} that is in many ways more powerful than the extension in previously published work \citet{yuen2000applications} and \citet{yuen2002generalization}.

\section{Introduction to Nonstandard Analysis} \label{SecIntroNSA}
We briefly introduce the setting and notation from nonstandard analysis.
For those who are not familiar with nonstandard analysis, \citet{Markovpaper} and \citet{nsbayes} provide introduction tailored to statisticians and probabilists. \citet{NSAA97,NDV,NAW} provide thorough introductions.

We use $\NSE{}$ to denote the nonstandard extension map taking elements, sets, functions, relations, etc., to their nonstandard counterparts.
In particular, $\HReals$ and $\NSE{\Nats}$ denote the nonstandard extensions of the reals and natural numbers, respectively.
An element $r\in \HReals$ is \emph{infinite} if $|r|>n$ for every $n\in \Nats$ and is \emph{finite} otherwise. An element $r \in \HReals$ with $r > 0$ is \textit{infinitesimal} if $r^{-1}$ is infinite. For $r,s \in \HReals$, we use the notation $r \approx s$ as shorthand for the statement ``$|r-s|$ is infinitesimal," and similarly we use use $r \gtrapprox s$ as shorthand for the statement ``either $r \geq s$ or $r \approx s$."

Given a topological space $(X,\topology)$,
the monad of a point $x\in X$ is the set $\bigcap_{ U\in \topology \, : \, x \in U}\NSE{U}$.
An element $x\in \NSE{X}$ is \emph{near-standard} if it is in the monad of some $y\in X$.
We say $y$ is the standard part of $x$ and write $y=\ST(x)$. Note that such $y$ is unique.
We use $\NS{\NSE{X}}$ to denote the collection of near-standard elements of $\NSE{X}$
and we say $\NS{\NSE{X}}$ is the \emph{near-standard part} of $\NSE{X}$.
The standard part map $\ST$ is a function from $\NS{\NSE{X}}$ to $X$, taking near-standard elements to their standard parts.
In both cases, the notation elides the underlying space $Y$ and the topology $T$,
because the space and topology will always be clear from context.
For a metric space $(X,d)$, two elements $x,y\in \NSE{X}$ are \emph{infinitely close} if $\NSE{d}(x,y)\approx 0$.
An element $x\in \NSE{X}$ is near-standard if and only if it is infinitely close to some $y\in X$.
An element $x\in \NSE{X}$ is finite if there exists $y\in X$ such that $\NSE{d}(x,y)<\infty$ and is infinite otherwise.

Let $X$ be a topological space endowed with Borel $\sigma$-algebra $\BorelSets X$. An internal probability measure $\mu$ on $(\NSE{X},\NSE{\BorelSets X})$
is an internal function from $\NSE{\BorelSets X}\to \NSE{[0,1]}$ such that
\begin{enumerate}
\item $\mu(\emptyset)=0$;
\item $\mu(\NSE{X})=1$; and
\item $\mu$ is hyperfinitely additive.
\end{enumerate}
The Loeb space of the internal probability space $(\NSE{X},\NSE{\BorelSets X}, \mu)$ is a countably additive probability space $(\NSE{X},\Loeb{\NSE{\BorelSets X}}, \Loeb{\mu})$ such that
\[
\Loeb{\NSE{\BorelSets X}}=\{A\subset \NSE{X}|(\forall \epsilon>0)(\exists A_i,A_o\in \NSE{\BorelSets X})(A_i\subset A\subset A_o\wedge \mu(A_o\setminus A_i)<\epsilon)\}
\]
and
\[
\Loeb{\mu}(A)=\sup\{\ST(\mu(A_i))|A_i\subset A,A_i\in \NSE{\BorelSets X}\}=\inf\{\ST(\mu(A_o))|A_o\supset A,A_o\in \NSE{\BorelSets X}\}.
\]

Every standard model is closely connected to its nonstandard extension via the \emph{transfer principle}, which asserts that a first order statement is true in the standard model is true if and only if it is true in the nonstandard model.
Finally, given a cardinal number $\kappa$, a nonstandard model is called $\kappa$-saturated if the following condition holds:
let $\cF$ be a family of internal sets, if $\cF$ has cardinality less than $\kappa$ and $\cF$ has the finite intersection property, then the total intersection of $\cF$ is non-empty. In this paper, we assume our nonstandard model is as saturated as we need (see \textit{e.g.} {\citep[][Thm.~1.7.3]{NSAA97}} for the existence of $\kappa$-saturated nonstandard models for any uncountable cardinal $\kappa$).

\section{Hyperfinite Markov Processes}\label{hyprocess}
We start this section by giving an overview of hyperfinite Markov processes developed in \citep{Markovpaper}, \citep{anderson2018mixhit} and \citep{anderson2019gibbs}.  
Intuitively, hyperfinite Markov processes behave like finite Markov processes but can be used to represent general Markov processes under moderate conditions. For the remainder of this paper, we assume that a probability space $X$ is always endowed with Borel $\sigma$-algebra $\BorelSets X$ unless otherwise mentioned.  

\begin{definition}\label{defnHM}
A general hyperfinite Markov chain on $\NSE{X}$ is characterized by the following four ingredients:

\begin{enumerate}

\item A hyperfinite state space $S\subset {^{*}X}$.

\item A hyperfinite time line $T=\NSE{\Nats}$.

\item A set $\{v_i:i\in S\}$ of non-negative hyperreals such that $\sum_{i\in S}v_i=1$.

\item A set $\{G_{ij}\}_{i,j\in S}$ consisting of non-negative hyperreals with $\sum_{j\in S}G_{ij}=1$ for each $i\in S$

\end{enumerate}
\end{definition} 

The state space $S$ naturally inherits the $^{*}$metric of ${^{*}X}$. For every $i,j\in S$, $G_{ij}$ refers to the internal probability of going from $i$ to $j$. The following theorem shows the existence of hyperfinite Markov process. 

\begin{theorem}[{\citep[][Thm.~7.2]{Markovpaper}}]\label{HMexist}
Given a non-empty hyperfinite state space $S\subset \NSE{X}$, $\{v_i\}_{i\in S}$ and $\{G_{ij}\}_{i,j\in S}$ as in \cref{defnHM}. Then there exists a $\NSE{\sigma}$-additive probability triple $(\Omega,\cA,P)$ with an internal stochastic process $\{X_t\}_{t\in T}$ defined on $(\Omega,\cA,P)$ such that
\[
P(X_0=i_0,X_{1}=i_{1},...X_{t}=i_t)=v_{i_{0}}G_{i_{0}i_{1}}...G_{i_{t-1}i_{t}}
\]
for all $t\in \NSE{\Nats}$ and $i_0,....i_{t}\in S$. 
\end{theorem} 

The proof of \cref{HMexist} essentially follows from transferring the existence theorem for finite Markov processes. 
For every $i\in S$ and every internal $A\subset S$, define $G_{i}(A)=\sum_{j\in A}G_{ij}$. 
For $n\in T$, define $G_{i}^{(n+1)}(A)=\sum_{j\in S}G_{j}^{(n)}(A)G_{ij}$.
It is easy to see that $G_{i}^{(n)}(\cdot)$ is an internal probability measure on $S$ for every $i\in S$ and $n\in T$. 
We call $\{G_{i}^{(n)}(\cdot)\}$ the $n$-step internal transition kernel of the underlying hyperfinite Markov process. 
Just like standard Markov process, a hyperfinite Markov process is characterized by its $1$-step internal transition kernel.

As in \citep{Markovpaper}, \citep{anderson2018mixhit} and \citep{anderson2019gibbs}, we shall construct hyperfinite Markov processes from standard Markov processes. Let $\{g(x,1,\cdot)\}_{x\in X}$ be the transition kernel of a Markov process on $X$. We now define the hyperfinite representation of the state space $X$. 

\begin{definition}\label{hyperapproxsp}
Let $(X,d)$ be a compact metric space. Let $\delta\in \NSE{\PosReals}$ be an infinitesimal. A $\delta$-hyperfinite representation of $X$ is a tuple $(S,\{B(s)\}_{s\in S})$ such that
\begin{enumerate}
\item $S$ is a hyperfinite subset of $\NSE{X}$.
\item $s\in B(s)\in \NSE{\BorelSets X}$ for every $s\in S$ and $\bigcup_{s\in S}B(s)=\NSE{X}$.
\item For every $s\in S$, the diameter of $B(s)$ is no greater than $\delta$.
\item $B(s_1)\cap B(s_2)=\emptyset$ for every $s_1\neq s_2\in S$.
\end{enumerate}
The set $S$ is called the \emph{base set} of the hyperfinite representation. For every $x\in \NSE{X}$, we use $s_{x}$ to denote the unique element in $S$ such that $x\in B(s_x)$.
\end{definition}

We say $(S,\{B(s)\}_{s\in S})$ is a \emph{hyperfinite representation} of $X$ if $(S,\{B(s)\}_{s\in S})$ is a $\delta$-hyperfinite representation of $X$ for some infinitesimal $\delta$. Hyperfinite representations always exist. 

\begin{theorem}[{\citep[][Theorem.~3.3]{anderson2018mixhit}}]\label{exhyper}
Let $X$ be a compact metric space. Then, for every positive infinitesimal $\delta$, there exists an $\delta$-hyperfinite representation $(S_{\delta},\{B(s)\}_{s\in S_{\delta}})$ of ${^{*}X}$.
\end{theorem} 

We fix such a hyperfinite representation of $\NSE{X}$ and denoted it by $S$ in the remainder of this section. 
We now define a hyperfinite Markov process transition kernel on $S$ from $\{g(x,1,\cdot)\}_{x\in X}$. For $i,j\in S$, define by $G_{ij}=\NSE{g}(i,1,B(j))$. For every internal set $A\subset S$, define $G_{i}(A)=\sum_{j\in A}G_{ij}$. It is straightforward to check that $\{G_{ij}\}_{i,j\in S}$ defines the one-step internal transition kernel for some hyperfinite Markov process. In order to establish a connection between $\{G_{ij}\}_{i,j\in S}$ and $\{g(x,1,\cdot)\}_{x\in X}$, we impose the following nonstandard condition on $\{\NSE{g}(x,1,\cdot)\}_{x\in \NSE{X}}$. 

\begin{condition}{SSF}\label{assumptionwsf}
Let $(S, \{B(s)\}_{s\in S})$ be a hyperfinite representation of $X$. 
The transition kernel $\{g(x,1,\cdot)\}_{x\in X}$ satisfies the \emph{S-strong Feller property} if, for every $x\in \NSE{X}$ and every internal $A\subset S$, we have 
\[
|\NSE{g}(x,1,\bigcup_{a\in A}B(a))-\NSE{g}(s_x,1,\bigcup_{a\in A}B(a))|\approx 0,
\]
where $s_x$ is the unique point in $S$ with $x\in B(s_x)$. 
\end{condition}

For two (internal) probability measures $P_1$ and $P_2$, we use $\parallel P_1-P_2 \parallel$ to denote the *total variational distance between $P_1$ and $P_2$. \cref{assumptionwsf} is a consequence of the following classical condition on $\{g(x,1,\cdot)\}_{x\in X}$. 

\begin{condition}{DSF}\label{assumptiondsf}
The transition kernel $\{g(x,1,\cdot)\}_{x\in X}$ satisfies the \emph{strong Feller property} if for every $x\in X$ and every $\epsilon>0$ there exists $ \delta>0$ such that
\[
(\forall y\in X)(|y-x|<\delta \implies \parallel g(x,1,A)-g(y,1,A)\parallel <\epsilon).
\]
\end{condition}

Suppose $\{g(x,1,\cdot)\}_{x\in X}$ is a reversible transition kernel with stationary distribution $\pi$. Define $\Pi$ on $(S,\{B(s)\}_{s\in S})$ by letting $\Pi(\{a\})=\NSE{\pi}(B(a))$ for every $a\in S$. The internal transition kernel $\{G_{i}(\cdot)\}_{i\in S}$ may not be
*reversible with respect to $\Pi$. In the following theorem, we construct an internal transition kernel $\{H_{i}(\cdot)\}_{i\in S}$, which is infinitesimally close to $\{G_{i}(\cdot)\}_{i\in S}$ and is *reversible with respect to $\Pi$. 

\begin{theorem}[{\citep[][Thm.~3.7]{anderson2019gibbs}}]\label{closereverse}
Let $\{g(x,1,\cdot)\}_{x\in X}$ be a transition kernel on $X$. 
Suppose $\{g(x,1,\cdot)\}_{x\in X}$ is reversible with stationary measure $\pi$ and satisfies \cref{assumptionwsf}.
For every $i,j\in S$, define
\[
H_{i}(\{j\})=\frac{\int_{B(i)}\NSE{g}(x,1,B(j))\NSE{\pi}(\dee x)}{\NSE{\pi}(B(i))}
\]
if $\NSE{\pi}(B(i))\neq 0$. Define $H_{ij}=G_{ij}=\NSE{g}(i,1,B(j))$ if $\NSE{\pi}(B(i))=0$.
\begin{enumerate}
\item $\{H_{i}(\cdot)\}_{i\in S}$ defines an one-step internal transition kernel. 
\item $\Pi$ is a *stationary distribution for $\{H_{i}(\cdot)\}_{i\in S}$. 
\item $\{H_{i}(\cdot)\}_{i\in S}$ is *reversible with respect to $\Pi$.
\item $\max_{i\in S}\parallel G^{(t)}_{i}(\cdot)-H^{(t)}_{i}(\cdot) \parallel\approx 0$ for every $t\in \Nats$.
\end{enumerate}
\end{theorem}

We shall use notations $\{g(x,1,\cdot)\}_{x\in X}$, $\{G_{i}(\cdot)\}_{i\in S}$ and $\{H_{i}(\cdot)\}_{i\in S}$ as they defined in this section for the rest of the paper.

\section{Hyperfinite Representation of Dirichlet Form}

In this section, we study the relationship between Dirichlet forms for general discrete-time Markov processes and Dirichlet forms for hyperfinite Markov processes. 

Let $\{g(x,1,\cdot)\}_{x\in X}$ denote the transition kernel of a Markov process and let $\pi$ denote its stationary distribution. 
the \emph{Dirichlet form} associated with $g$ of a function $f: X\to \Reals$ is defined to be 
\[
\Diric{g}{f}{f}=\frac{1}{2}\int_{x\in X}\int_{y\in X}[f(x)-f(y)]^{2}g(x,1,\dee y)\pi(\dee x).
\]
When the state space $X$ is finite, the integral in the above equation becomes summation.

For the remainder of this section, assume that the state space $X$ is compact. Let $(S,\{B(s)\}_{s\in S})$ denote a hyperfinite representation of $X$. Let $\{g(x,1,\cdot)\}_{x\in X}$ be a transition kernel on $X$ with stationary distribution $\pi$. We define $\{G_{i}(\cdot)\}_{i\in S}$, $\{H_{i}(\cdot)\}_{i\in S}$ and $\Pi$ as in \cref{hyprocess}. We now use hyperfinite Dirichlet forms to approximate standard Dirichlet forms. We first quote the following well-known result in nonstandard analysis.

\begin{theorem}[{\citep[][Thm.~5.3]{NDV}}]\label{sctsfull}
Let $(X,\BorelSets X,\mu)$ be a Radon probability space with Borel $\sigma$-algebra and let $f: X\to \Reals$ be measurable. 
Then $\NSE{f}$ is a lifting of $f$ with respect to $\overline{\NSE{\mu}}$ i.e.
\[
\NSE{f}(x)\approx f(\ST(x))
\]
for $\overline{\NSE{\mu}}$-almost all $x\in \NSE{X}$. Consequently, there is a set $Y\subset \NSE{X}$ with $\overline{\NSE{\mu}}(Y)=1$ such that for all $y_1,y_2\in Y$
\[
(y_1\approx y_2)\implies (\NSE{f}(y_1)\approx \NSE{f}(y_2)).
\]
\end{theorem}

Before proving the main result of this section, we quote the following useful lemma.
\begin{lemma}[{\citep[][Lemma.~7.24]{Markovpaper}}] \label{tvfunction}
Let $P_1$ and $P_2$ be two internal probability measures on a hyperfinite set $S$. Then
\[
\parallel P_1(\cdot)-P_2(\cdot) \parallel \geq \NSE{\sup}_{f: S\to {^{*}[0,1]}}|\sum_{i\in S}P_{1}(\{i\})f(i)-\sum_{i\in S}P_{2}(\{i\})f(i)|,
\]
where $\parallel P_1(\cdot)-P_2(\cdot) \parallel=\NSE{\sup}_{A\in \mathcal{I}(S)}|P_{1}(A)-P_{2}(A)|$ and the function $f$ ranges over all internal functions. 
\end{lemma}

Our main result in this section is:
\begin{theorem}\label{hydstd}
Suppose $X$ is a compact metric space.
Let $\{U_1,U_2,\dotsc,U_{K}\}\subset \NSE{\BorelSets X}$ be a hyperfinite partition of $\NSE{X}$ such that the diameter of $U_j$ is infinitesimal for each $j\leq K$. 
Then there exists a hyperfinite representation $(S_X,\{B(s):s\in S_X\})$ of $X$ such that 
\begin{enumerate}
\item $S_X$ has internal cardinality $K$ and $\{B(s): s\in S_X\}=\{U_j: j\leq K\}$.
\item Suppose \cref{assumptionwsf} holds, then for every continuous function $f: X\to \Reals$, we have
\[
\Diric{g}{f}{f}\approx \frac{1}{2}\sum_{s,t\in S_{X}}[F(s)-F(t)]^{2}H_{s}(\{t\})\Pi(\{s\})
\]
where $F: S_X\to \HReals$ is defined as $F(s)=\NSE{f}(s)$ for every $s\in S_X$. 
\end{enumerate}
\end{theorem}
\begin{proof}
Let $\{U_j: j\leq K\}$ be a partition of $\NSE{X}$ of *Borel sets with infinitesimal radius and let $k\in \Nats$.
Pick $n\in \Nats$ and let $f_1,\dotsc,f_n$ be $n$ continuous functions from $X$ to $\Reals$. Let 
\[\label{fuppbd}
a=\max\{[f_{i}(x)-f_{i}(y)]^{2}: x,y\in X, 1\leq i\leq n\}\in \Reals.
\]
For $1\leq i\leq n$, let $r_{i}(x)=\int_{y\in X}[f_{i}(x)-f_{i}(y)]^{2}g(x,1,\dee y)$. As $r_{i}$ is a measurable function from $X\to \Reals$, by \cref{sctsfull}, there exists a Loeb measurable set $M_i$ with $\Loeb{\NSE{\pi}}(M_i)=1$ such that $\NSE{r_i}(x)\approx r_{i}(\ST(x))$ for $x\in M_i$. Let $M=\bigcap_{i=1}^{n}M_i$. Note that $\Loeb{\NSE{\pi}}(M)=1$. Thus, $\NSE{r_k}(y)\approx r_{k}(\ST(y))$ for all $1\leq k\leq n$ and all $y\in M$. 

Let $Y\subset M$ be a *Borel set such that $\NSE{\pi}(Y)>1-\frac{1}{ka}$. For $1\leq j\leq K$, if $U_j\cap Y=\emptyset$, we pick an arbitrary element $s_j\in U_j$. If $U_j\cap Y\neq \emptyset$, we pick an arbitrary element $s_j\in U_j\cap Y$. By the internal definition principle, we know that $S=\{s_j: j\leq K\}$ is a hyperfinite subset of $\NSE{X}$. 
For every $s\in S$, we use $B(s)$ to denote the unique element in $\{U_j:j\leq K\}$ which contains $s$.
It is easy to see that $(S, \{B(s): s\in S\})$ is a hyperfinite representation of $\NSE{X}$. 
Moreover, $S$ has internal cardinality $K$ and $\{B(s): s\in S\}=\{U_j: j\leq K\}$. 
For every $1\leq i\leq n$, let $F_i: S\to \NSE{\Reals}$ be $F_{i}(s)=\NSE{f_{i}}(s)$. For every $s\in S$ and every $1\leq i\leq n$, let
\[
R_{i}(s)=\sum_{t\in S}[F_{i}(s)-F_{i}(t)]^{2}H_{s}(\{t\}).
\]
Let $Z=\{s\in S: B(s)\cap Y\neq \emptyset\}$. It is straightforward to see that $\Pi(Z)=\NSE{\pi}(\bigcup_{s\in Z}B(s))\geq \NSE{\pi}(Y)>1-\frac{1}{ka}$. 

\begin{claim}\label{zapproxclaim}
For every $z\in Z$ and every $1\leq i\leq n$, we have $|R_{i}(z)-\NSE{r}_{i}(z)|\lessapprox \frac{2}{k}$ and $\NSE{r_i}(z)\approx r_{i}(\ST(z))$. 
\end{claim}
\begin{proof}
Pick some $1\leq i\leq n$ and pick some $z\in Z$. 
By construction, we have $z\in Y$. By \cref{sctsfull}, we have $\NSE{r_{i}}(z)\approx r_{i}(\ST(z))$. 
Let $Y_{i}\supset Y$ be a *Borel subset of $\NSE{X}$ such that $\NSE{g}(z,1,Y_{i})>1-\frac{1}{ka}$.  
We have
\[
\NSE{r_{i}}(z)&=\int_{x\in \NSE{X}}[(\NSE{f_i}(z)-\NSE{f_i}(x))^{2}]\NSE{g}(z,1,\dee x)\\
&=\int_{x\in Y_{i}}[(\NSE{f_i}(z)-\NSE{f_i}(x))^{2}]\NSE{g}(z,1,\dee x)+\int_{x\in \NSE{X}\setminus Y_{i}}[(\NSE{f_i}(z)-\NSE{f_i}(x))^{2}]\NSE{g}(z,1,\dee x).
\]
Note that $\int_{x\in \NSE{X}\setminus Y_{i}}[(\NSE{f_i}(z)-\NSE{f_i}(x))^{2}]\NSE{g}(z,1,\dee x)<\frac{1}{k}$.
Let 
\[
Z_{i}=\{s\in S: B(s)\cap Y_{i}\neq \emptyset\}. 
\]
Note that $G_{z}(Z_{i})>1-\frac{1}{ka}$ and $a$ in \cref{fuppbd} is an upper bound of
$[F_{i}(z)-F_{i}(s)]^2$. Then we have
\[
&\sum_{s\in S}[F_{i}(z)-F_{i}(s)]^{2}G_{z}(\{s\})\\
&=\sum_{s\in Z_{i}}[F_{i}(z)-F_{i}(s)]^{2}G_{z}(\{s\})+\sum_{s\in S\setminus Z_{i}}[F_{i}(z)-F_{i}(s)]^{2}G_{z}(\{s\}). 
\]
Note that $\sum_{s\in S\setminus Z_{i}}[F_{i}(z)-F_{i}(s)]^{2}G_{z}(\{s\})<\frac{1}{k}$. 

By the continuity of $f_i$, we have
\[
&\int_{x\in Y_{i}}[(\NSE{f_i}(z)-\NSE{f_i}(x))^{2}]\NSE{g}(z,1,\dee x)\\
&=\sum_{s\in Z_{i}}\int_{x\in B(s)\cap Y_{i}}[(\NSE{f_i}(z)-\NSE{f_i}(x))^{2}]\NSE{g}(z,1,\dee x)\\
&\approx \sum_{s\in Z_{i}}\int_{x\in B(s)\cap Y_{i}}[(\NSE{f_i}(z)-\NSE{f_i}(s))^{2}]\NSE{g}(z,1,\dee x)\\
&= \sum_{s\in Z_{i}}[(F_{i}(z)-F_{i}(s))^2]\NSE{g}(z,1,B(s)\cap Y_{i}).
\]
As $\sum_{s\in Z_{i}}[(F_{i}(z)-F_{i}(s))^2]G_{z}(\{s\})-\sum_{s\in Z_{i}}[(F_{i}(z)-F_{i}(s))^2]\NSE{g}(z,1,B(s)\cap Y_{i})<\frac{1}{k}$, we have $|\NSE{r_{i}}(z)-\sum_{s\in S}[(F_{i}(z)-F_{i}(s))^2]G_{z}(\{s\})|\lessapprox \frac{2}{k}$.

By \cref{closereverse} and \cref{tvfunction}, we have
\[
\sum_{s\in S}[F_{i}(z)-F_{i}(s)]^{2}G_{z}(\{s\})=a\sum_{s\in S}\frac{[F_{i}(z)-F_{i}(s)]^{2}}{a}G_{z}(\{s\})\approx \sum_{s\in S}[F_{i}(z)-F_{i}(s)]^{2}H_{z}(\{s\})
\]
Hence, we have $|R_{i}(z)-\NSE{r}_{i}(z)|\lessapprox \frac{2}{k}$ for every $z\in Z$ and every $1\leq i\leq n$. 
\end{proof}

For $1\leq i\leq n$, let $\cR_{S}(F_i,F_i)=\frac{1}{2}\sum_{s\in S}R_{i}(s)\Pi(\{s\})$. 

\begin{claim}
For $1\leq i\leq n$, we have $|\Diric{g}{f_i}{f_i}-\cR_{S}(F_i,F_i)|<\frac{5}{k}$. 
\end{claim}
\begin{proof}
Pick some $1\leq i\leq n$. 
By definition, we have 
\[
2\Diric{g}{f_i}{f_i}&=\int_{x\in X}r_{i}(x)\pi(\dee x)\\
&=\int_{x\in \NSE{X}}\NSE{r_i}(x)\NSE{\pi}(\dee x)\\
&=\int_{x\in Y}\NSE{r_i}(x)\NSE{\pi}(\dee x)+\int_{x\in \NSE{X}\setminus Y}\NSE{r_i}(x)\NSE{\pi}(\dee x). 
\]
Note that $\int_{x\in \NSE{X}\setminus Y}\NSE{r_i}(x)\NSE{\pi}(\dee x)<\frac{1}{k}$. 

On the other hand, we have
\[
2\cR_{S}(F_i,F_i)&=\sum_{s\in S}R_{i}(s)\Pi(\{s\})\\
&=\sum_{s\in Z}R_{i}(s)\Pi(\{s\})+\sum_{s\in S\setminus Z}R_{i}(s)\Pi(\{s\}).
\]
Note that $\sum_{s\in S\setminus Z}R_{i}(s)\Pi(\{s\})<\frac{1}{k}$. 

We now compare terms $\int_{x\in Y}\NSE{r_i}(x)\NSE{\pi}(\dee x)$ and $\sum_{s\in Z}R_{i}(s)\Pi(\{s\})$. 
Note that $\bigcup_{s\in Z}B(s)\supset Y$ and $\NSE{r_i}$ is S-continuous in $Y$. By \cref{zapproxclaim}, we have
\[
&|\int_{x\in Y}\NSE{r_i}(x)\NSE{\pi}(\dee x)-\sum_{s\in Z}R_{i}(s)\NSE{\pi}(B(s)\cap Y)|\\
&=|\sum_{s\in Z}\int_{x\in B(s)\cap Y}\NSE{r_i}(x)\NSE{\pi}(\dee x)-\sum_{s\in Z}R_{i}(s)\NSE{\pi}(B(s)\cap Y)|\\
&\approx |\sum_{s\in Z}\NSE{r_i}(s)\NSE{\pi}(B(s)\cap Y)-\sum_{s\in Z}R_{i}(s)\NSE{\pi}(B(s)\cap Y)|\lessapprox \frac{2}{k}\\
\]
Note that
\[
&\sum_{s\in Z}R_{i}(s)\Pi(\{s\})-\sum_{s\in Z}R_{i}(s)\NSE{\pi}(B(s)\cap Y)\\
&=\sum_{s\in Z}R_{i}(s)[\NSE{\pi}(B(s))-\NSE{\pi}(B(s)\cap Y)]<\frac{1}{k}.
\]
Thus, we have $|\Diric{g}{f_i}{f_i}-\cR_{S}(F_i,F_i)|\lessapprox\frac{4}{k}<\frac{5}{k}$.
\end{proof}

Let $\cC$ denote the collection of all continuous functions from $X$ to $\Reals$ and let $\phi_{f}^{k}(S)$ be the conjunction of the following formulas:
\begin{enumerate}
\item $(S,\{U_j:j\leq K\})$ is a hyperfinite representation of $\NSE{X}$. 
\item $|\NSE{\Diric{g}{f}{f}}-\frac{1}{2}\sum_{s,t\in S}[F(s)-F(t)]^{2}G_{s}(\{t\})\Pi(\{s\})|<\frac{5}{k}$ where $F(s)=\NSE{f}(s)$ 
for every $s\in S$. 
\end{enumerate} 
The family $\cB=\{\phi_{f}^{k}: f\in \cC, k\in \Nats\}$ is finitely satisfiable. By saturation, there exists $S_X$ such that every formula in $\cB$ is satisfied simultaneously. Such $(S_X,\{U_j:j\leq K\})$ is the desired hyperfinite representation.  
\end{proof}

The \emph{hyperfinite Dirichlet form} with $H$ of an internal function $F: S_X\to \NSE{\Reals}$ is defined to be
\[
\hyDiric{H}{F}{F}=\frac{1}{2}\sum_{s,t\in S_{X}}[F(s)-F(t)]^{2}H_{s}(\{t\})\Pi(\{s\}).
\]
The result in this section shows that, under moderate conditions, we can use hyperfinite Dirichlet form to approximate the standard Dirichlet form. 

\section{Nonstandard Characterization of Density Functions}\label{secnsdensity}
In this section, we study density functions of transition kernels and stationary distributions via studying their corresponding hyperfinite counterparts. The materials presented in this section are closely related to \citep{zimradon}. We start by introducing the notion of S-integrability from nonstandard analysis. 

\begin{definition}\label{defsint}
Let $(\Omega,\cA,P)$ be an internal probability space 
and let $F: \Omega\to {^{*}\Reals}$ be an internally integrable function such that $\ST(F)$ exists $\Loeb{P}$-almost surely. 
Then $F$ is \emph{S-integrable} with respect to $P$  if $\ST({F})$ is $\Loeb{P}$-integrable, and ${^{*}\int F\dee P}\approx \int \ST({F})\dee \overline{P}$.
\end{definition}

The following theorem provides several ways to verify the S-integrability of an internal function $F$. 

\begin{theorem} [{\citep[][Theorem~4.6.2]{NSAA97}}]\label{Lintegral} 
Suppose $(\Omega,\cA,P)$ is an internal probability space, and $F: \Omega\to {^{*}\Reals}$ is an internally integrable function such that $\ST({F})$ exists $\Loeb{P}$-almost surely. Then the following are equivalent:
\begin{enumerate}
\item $\ST({\int |F|\dee P})$ exists and it equals to $\lim_{n\to \infty}\ST({\int |F_n|\dee P})$ where for $n\in \Nats$, $F_n=\min\{F,n\}$ when $F\geq 0$ and $F_n=\max\{F,-n\}$ when $F\leq 0$.

\item For every infinite $K>0$, $\int_{|F|>K}|F|\dee P\approx 0$.

\item $\ST({\int |F|\dee P})$ exists, and for every $B$ with $P(B)\approx 0$, we have $\int_{B}|F|\dee P\approx 0$.

\item $F$ is S-integrable with respect to $P$. 
\end{enumerate}
\end{theorem}

For the rest of this section, let $Y$ be a compact metric space and let $\{A_i: 1\leq i\leq T\}$ be a hyperfinite partition of $\NSE{Y}$ such that the diameter for $A_i$ is infinitesimal for every $i\leq T$. 

\begin{definition}
Let $P, Q$ be two probability measures on $(Y,\BorelSets Y)$. 
$Q$ dominates $P$ if for each $\epsilon>0$, there exists $\delta>0$ such that $P(C)<\epsilon$ for all $C\in \BorelSets Y$ with $Q(C)<\delta$. 
Let $\cP$ be a family of probability measures on $(Y,\BorelSets Y)$. Then $\cP$ is uniformly dominated by $Q$ if, for every $\epsilon>0$, there exists $\delta>0$ such that $P(C)<\epsilon$ for all $P\in \cP$ and all $C\in \BorelSets Y$ with $Q(C)<\delta$. 
\end{definition}

Using the nonstandard characterization of domination, $Q$ dominates $P$ if and only if $\NSE{Q}(A)\approx 0$ implies $\NSE{P}(A)\approx 0$ for $A\in \NSE{\BorelSets Y}$. We now mimic proofs in \citep{zimradon} to give an useful nonstandard characterization of density functions. 

\begin{lemma}[{\citep[][Lemma.~2.1]{zimradon}}]\label{hyradonint}
Let $P,Q$ be two probability measures on $(Y,\BorelSets Y)$.
Suppose $P$ is dominated by $Q$. 
Then the function $\phi: \NSE{Y}\to \NSE{\Reals}$ defined by 
\[\label{defhyradon}
\phi(x)=\sum_{i=1}^{T}\frac{\NSE{P}(A_i)}{\NSE{Q}(A_i)}\indicate_{A_i}(x)
\]
is S-integrable with respect to $Q$.
\end{lemma}
\begin{proof}
It is straightforward to verify that $\int_{\NSE{Y}}\phi(x) \NSE{Q}(\dee x)=1$.
As $\phi(x)\geq 0$, we know that $\phi(x)\in \NS{\NSE{\Reals}}$ for $\Loeb{\NSE{Q}}$-almost all $x\in X$. 
Pick a set $A\in \NSE{\BorelSets Y}$ such that $\NSE{Q}(A)\approx 0$. 
By domination, we know that $\NSE{P}(A) \approx 0$. Thus, we have
\[
\int_{A} \phi(x) \NSE{Q}(\dee x)=\sum_{i=1}^{T}\NSE{P}(A\cap A_i)\leq \NSE{P}(A)\approx 0.
\]
Thus, by \cref{Lintegral}, $\phi$ is S-integrable with respect to $Q$. 
\end{proof} 

We now show that $\phi$ defined in \cref{hyradonint} is infinitesimally close to the standard Radon-Nidodym derivative. 

\begin{theorem}\label{stradon}
Let $P,Q$ be two probability measures on $(Y,\BorelSets Y)$.
Suppose $P$ is dominated by $Q$.
Let $f$ denote the density function of $P$ with respect to $Q$.
Suppose $f$ is continuous at $y_0\in Y$. 
Let $\phi$ be the internal function defined in \cref{defhyradon}. Then $\phi(x)\approx \NSE{f}(y)$ for $x,y\in \NSE{Y}$ with $x\approx y\approx y_0$. 
\end{theorem}
\begin{proof}
Pick $A_i$ such that every point in $A_i$ is infinitely close to $y_0$. 
Note that $\phi$ is constant on $A_i$. 
By the transfer principle, we know that $\NSE{P}(A_i)=\int_{A_i}\NSE{f}(x) \NSE{Q}(\dee x)$ for every $1\leq i\leq T$. 
We also know that $\NSE{P}(A_i)=\int_{A_i} \phi(x) \NSE{Q}(\dee x)$. 
So we have $\int_{A_i} [\NSE{f}(x)-\phi(x)]\NSE{Q}(\dee x)=0$. 
As $\phi$ is constant on $A_i$ and $f$ is continuous at $y_0$, we know that $\NSE{f}(x)\approx \phi(x)$ for all $x\in A_i$. 
By the continuity of $f$ at $y_0$ again, we have $\NSE{f}(x)\approx \phi(y)$ for all $x,y\in \NSE{Y}$ with $x\approx y\approx y_0$.
\end{proof}

\cref{stradon} implies that $\phi(x)=f(\ST(x))$ provided that the density function $f$ is continuous at $\ST(x)$. 
In fact, by \cref{hyradonint} and \cref{stradon}, it can be shown that $\ST(\phi(x))$ ($x\in Y$) is a density function with respect to $Q$ (\citep[][Thm.~3.6]{zimradon}) provided that $\ST(\phi)$ is constant on monads.  

\subsection{Density Functions For Markov Processes}\label{sectiondfmarkov}
Let $\{g(x,1,\cdot)\}_{x\in X}$ be a Markov transition kernel with stationary distribution $\pi$ on compact metric space $X$ endowed with Borel $\sigma$-algebra $\BorelSets X$. Let $(S_X, \{B(s): s\in S_X\})$ be a hyperfinite representation of $X$. We use $G, H,\Pi$ to denote the corresponding hyperfinite objects as defined in \cref{hyprocess}. 

\begin{assumption}\label{assumptionde}
There exists a referencing measure $\lambda$ on $(X, \BorelSets X)$ 
such that $\pi$ is dominated by $\lambda$ and the family $\{g(x,1,\cdot)\}_{x\in X}$ is dominated by $\lambda$.
\end{assumption}

We use $f(x,\cdot)$ to denote a density function of $g(x,1,\cdot)$ with respect to $\lambda$ and use $k(\cdot)$ to denote the density function of $\pi$ with respect to $\lambda$. 
%
%

\begin{theorem}\label{markovdensity}
Suppose \cref{assumptionde} holds.
Suppose $k$ is continuous at $x_1\in X$ and $f$ is jointly continuous at $(x_1, x_2)$ for some $x_2\in X$.                     
Then, for $i,j\in S_X$ such that $i\approx x_1$ and $j\approx x_2$, we have $\frac{\Pi(\{i\})}{\NSE{\lambda}(B(i))}\approx \NSE{k}(i)$ and
$\frac{G_{i}(\{j\})}{\NSE{\lambda}(B(j))}\approx \NSE{f}(i,j)\approx \frac{H_{i}(\{j\})}{\NSE{\lambda}(B(j))}$. 
\end{theorem}

\begin{proof}
Clearly $\{B(s): s\in S_X\}$ is a hyperfinite partition of $\NSE{X}$ such that each $B(s)$ has infinitesimal diameter. 
As $k(\cdot)$ is continuous at $x_1$, by \cref{stradon}, we have
\[
\frac{\NSE{\pi}(B(i))}{\NSE{\lambda}(B(i))}=\frac{\Pi(\{i\})}{\NSE{\lambda}(B(i))}\approx \NSE{k}(i)
\] 
for all $i\approx x_1$. 

Pick $s,j\in S$ such that $s\approx x_1$ and $j\approx x_2$. By the transfer principle, we know that 
\[
G_{s}(\{j\})=\NSE{g}(s,1, B(j))=\int_{B(j)}\NSE{f}(s, x)\NSE{\lambda}(\dee x).
\]
On the other hand, we also know that 
\[
G_{s}(\{j\})=\NSE{g}(s, 1, B(j))=\int_{B(j)}\frac{G_{s}(\{j\})}{\NSE{\lambda}(B(j))}\NSE{\lambda}(\dee x).
\]
As $f$ is jointly continuous at $(x_1, x_2)$, we know that $\NSE{f}(s, x)\approx \NSE{f}(s, y)\approx f(\ST(s), \ST(j))$ for all $x, y\in B(j)$.
This implies that $\frac{G_{s}(\{j\})}{\NSE{\lambda}(B(j))}\approx \NSE{f}(s, j)$.

If $\NSE{\pi}(B(j))=0$, then 
\[
\frac{H_{s}(\{j\})}{\NSE{\lambda}(B(j))}=\frac{G_{s}(\{j\})}{\NSE{\lambda}(B(j))}\approx \NSE{f}(s, j).
\]
If $\NSE{\pi}(B(j))\neq 0$, by the joint continuity of $f$ at $(x_1, x_2)$, we have
\[
\frac{H_{s}(\{j\})}{\NSE{\lambda}(B(j))}&=\frac{\int_{B(s)}\frac{\NSE{g}(x,1,B(j))}{\NSE{\lambda}(B(j))}\NSE{\pi}(\dee x)}{\NSE{\pi}(B(s))}\\
&\approx \frac{\int_{B(s)}\NSE{f}(x,j)\NSE{\pi}(\dee x)}{\NSE{\pi}(B(s))}\approx \NSE{f}(s,j).
\]
Thus, we have the desired result. 

\end{proof}

\section{Comparison Theorem for General Markov Processes}\label{sec1dim}

Our main motivation for constructing a hyperfinite representation of the Dirichlet form is to allow us to directly translate theorems about transition kernels on finite-dimensional spaces to theorems on more general state space. As an illustration, we provide a simple translation of a ``comparison" theorem for Markov chains on finite state spaces (see the original Theorem \ref{finitecomparison} and our generalization Theorem \ref{mainthm} ). We also give a quick explanation, including an example, as to why \textit{direct} translation does not work well for this problem. These bounds will be more fully developed in our companion paper focused on comparison.

\subsection{Original  Comparison Bound }

We introduce comparison methods, following the same presentation as in \citep{markovmix}. For a finite reversible transition kernel $\{P(x,y)\}_{x,y\in \Omega}$ with stationary distribution $\pi$, let $E=\{(x,y): P(x,y)>0\}$. An $E$-path from $x$ to $y$ is a sequence $\Gamma=(e_1,e_2,\dotsc,e_n)$ of edges in $E$ such that $e_1=(x,x_1),e_2=(x_1,x_2),\dotsc, e_n=(x_{n-1},y)$ for some vertices $x_1,x_2,\dotsc,x_{n-1}\in \Omega$. We use $|\Gamma|$ to denote the length of a path $\Gamma$ and let $Q(x,y)=\pi(x)P(x,y)$ for $x,y\in \Omega$. 

Let $P$ and $\alt{P}$ be two finite reversible transition kernels with stationary distribution $\pi$ and $\alt{\pi}$, respectively. Suppose we can fix an $E$-path from $x$ to $y$ for every $(x,y)\in \alt{E}$ and denote it by $\gamma_{xy}$. Given such a choice of path, we define the \emph{congestion ratio} by 
\[\label{defepath}
B=\max_{e\in E} \left( \frac{1}{Q(e)}\sum_{(x,y): e\in \gamma_{xy}}\alt{Q}(x,y)|\gamma_{xy}| \right). 
\]

\begin{theorem}[{\citep[][Thm.~13.23]{markovmix}}]\label{finitecomparison}
Let $P$ and $\alt{P}$ be finite reversible transition kernels with stationary distributions $\pi$ and $\alt{\pi}$ respectively. If $B$ is the congestion ratio for a choice of $E$-paths, as defined in \cref{defepath}, then
\[
\Diric{\alt{P}}{f}{f}\leq B\Diric{P}{f}{f}.
\]
\end{theorem}

The aim of this section is to extend  \cref{finitecomparison} to a large class of general Markov Processes.

\subsection{Non-Example: Barbell Graph} \label{SecBarbellEx1}

Before giving our new theorem, we give a quick example that illustrates why \textit{directly} mimicking the form of the theorem for finite Markov chains will not work well. To understand the basic issue, note that we can view the bound in \cref{finitecomparison} as coming from three terms: the length $|\gamma_{xy}|$ of a path, the relative probability $\frac{\alt{Q}(x,y)}{Q(e)}$ of transitions along the path, and the congestion $|\{(x,y) \, : \, e \in \gamma_{xy}\}|$. The length makes sense as written for general state spaces, and the relative probability can easily be ``translated" to \textit{e.g.} the Radon-Nikodym derivative. However, it is not clear how to translate the notion of congestion, which measures how many paths go \textit{exactly} through an edge. The problem is that it is often simple to very slightly deform paths in continuous state spaces so that the congestion remains extremely small (or even 1), even though a very large number of paths are extremely close to an edge.

We give a concrete example. Denote by $K_{n}$ the usual complete graph on $n$ vertices and let $G_{n} = (V_{n},E_{n})$ be the ``barbell" graph obtained by taking two copies $R_{n},L_{n}$ of $K_{n}$ and adding a single edge; denote the endpoints of this new edge by $r_{n} \in R_{n}$ and $\ell_{n} \in L_{n}$. Denote by $\{X_{t}\}$ the usual simple random walk on this graph, with kernel $A_{n}$. We study the following natural collection of paths:

\begin{enumerate}
\item For $x,y \in L_{n}$ (or $x,y \in R_{n}$), simply take the edge from $x$ to $y$.
\item For $x \in L_{n}$ and $y \notin L_{n}$, take the edge $(x,\ell_{n})$ then the edge $(\ell_{n},r_{n})$ and then the edge $(r_{n},y)$. At the end, delete any self-edges that may appear. 
\item For $x \in R_{n}$ and $y \notin R_{n}$, do the obvious analogue to path \textbf{(2)}.
\end{enumerate}

It is straightforward to check that:
\begin{enumerate}
\item All paths are of length at most 3,
\item All transition probabilities are either $n^{-1}$ or $(n-1)^{-1}$ (so always roughly $n^{-1})$,
\item The stationary measure at any point is either $\frac{n}{2( (n-1)^{2} + n)}$ or $\frac{n-1}{2( (n-1)^{2} + n)}$ (so again roughly $n^{-1}$), and
\item The edge $(r_{n},\ell_{n})$ appears in roughly $n^{2}$ paths, which is (vastly) more than any other edge.
\end{enumerate}

Putting together these estimates and comparing to the kernel that simply takes i.i.d. samples from the stationary measure, \cref{finitecomparison} tells us that the relaxation time of $A_{n}$  is $O(n^{2})$ (and so the mixing time is $O(n^{2} \log(n))$). In fact it is straightforward to check that both the relaxation and mixing times are $\Theta(n^{2})$. Sketching the relevant arguments: the matching lower bound on the relaxation time can be obtained from Cheeger's inequality; upper and lower bounds on the mixing time can be obtained by checking that it takes $\Theta(n^{2})$ steps to travel from $L_{n} \backslash \{ \ell_{n}\}$ to $R_{n} \backslash \{ r_{n} \}$ on average (and, to obtain the upper bound, a straightforward coupling argument).

We now consider a continuous analogue. Define $S_{n} = \cup_{v \in V_{n}} I_{v}$, where $I_{v}$ are disjoint copies of $[0,1]$. For $v \in V_{n}$ and $y \in I_{v} \subset S_{n}$, write $V(y) = v$ for the vertex associated with $y$ and $|y|$ for the value of $y$ in $I_{v}$. Then define the transition kernel $Q_{n}$ on $S_{n}$ by the following algorithm for sampling $Y \sim Q_{n}(x,\cdot)$:

\begin{enumerate}
\item Sample $u \sim X_{n}(V(x),\cdot)$.
\item Return $Y \sim \mathrm{Unif}(I_{u})$.
\end{enumerate}

 Note that this walk can be written as a product walk for the original simple random walk on $G_{n}$ and the kernel whose measures are all $\mathrm{Unif}[0,1]$. In particular, this means it has the same mixing time and relaxation times as $\{X_{t}\}$. 

Attempting to confirm this with a naive version of \cref{finitecomparison} runs into an immediate problem. Roughly speaking, the problem comes from the fact that on continuous spaces we can ``split" the heavily-used vertices $\ell_{n}, r_{n}$ into many pieces, and this allows paths travelling between these vertices to avoid each other. Making this split precise requires some additional notation; the following somewhat-heavy notation is reused in Section \ref{SecBarbellEx2}. 

Let $N_{L,n} \, : \, L_{n} \backslash \{ \ell_{n} \} \mapsto \{1,2,\ldots,n-1\}$ be an arbitrary ordering of the vertices in $L_{n} \backslash \{ \ell_{n} \}$, and similarly let $N_{R,n}$ be an ordering of  $R_{n} \backslash \{ r_{n} \}$. For $u \in L_{n} \backslash \{\ell_{n}\}$, let $I_{\ell_{n}}(u)$ be a copy of $[0,1]$ and for $t \in [0,1]$ let $I_{\ell_{n}}(u,t) = t \in I_{\ell_{n}}(u)$. Let $I_{r_{n}}$ be the obvious analogous construction on $R_{n}$. Define 

\[ 
S_{n}' = \left(\sqcup_{v \in V_{n} \backslash\{\ell_{n},r_{n}\}} I_{v} \right) \sqcup \left( \sqcup_{v \in L_{n} \backslash \{\ell_{n}\}} I_{\ell_{n}}(v) \right) \sqcup \left(\sqcup_{v \in R_{n} \backslash \{r_{n}\}} I_{r_{n}}(v) \right); 
\] 
that is, we take $S_{n}$ and split the interval $I_{\ell_{n}}$ into $n-1$ intervals $\sqcup_{v \in L_{n} \backslash \{\ell_{n}\}} I_{\ell_{n}}(u)$, and make the analogous split for $I_{r_{n}}$. Define the map $\phi \, : \, S_{n}' \mapsto S_{n}$ by the following equations:

\begin{align*}
\phi(y) &= y, \qquad \qquad \qquad \qquad \qquad \qquad \, \,  V(y) \in V_{n} \backslash \{ \ell_{n}, r_{n} \} \\
V(\phi(y)) &= \ell_{n}, \qquad \qquad \qquad \qquad \qquad \qquad y \in \cup_{v \in L_{n} \backslash \{ \ell_{n} \}} I_{\ell_{n}}(v) \\
|\phi(y)| &= \frac{|y|}{n-1} + \frac{N_{n}(v) - 1}{n-1}, \qquad \qquad \, \, y \in L_{n}(v), 
\end{align*}
with the obvious analogues to the last two equations when $ y \in \cup_{v \in R_{n} \backslash \{ r_{n} \}} I_{r_{n}}(v)$. We note that, ignoring a finite number of points, $\phi$ is a bijection between $S_{n}$ and $S_{n}'$; ignoring events of measure 0, $Y_{t}' \equiv \phi^{-1}(Y_{t})$ defines a Markov chain with the same relaxation and mixing times as $\{Y_{t}\}$. For the remainder of this section we will consider this process on $S_{n}'$, and write $Q_{n}'$ for its kernel.

We now consider the following path from $x \in I_{x}$ to $y \in I_{y}$:

\begin{enumerate}
\item If $V(x),V(y) \in L_{n}$ (or $V(x),V(y) \in R_{n}$), simply take the edge from $x'$ to $y'$.
\item For $V(x) \in L_{n}$ and $V(y) \notin L_{n}$, take the edge from $x' \in I_{x}$ to $x'' \equiv I(\ell_{n},|x'|) \in I_{\ell_{n}}$, then the  edge from $x''$ to $y'' \equiv I(r_{n},|y'|)$, and then the edge from $y''$ to $y'$. At the end, delete any self-edges that may appear. 
\item For $V(x) \in R_{n}$ and $V(y) \notin R_{n}$, do the obvious analogue to path \textbf{(2)}.
\end{enumerate}

We again try to analyze these paths. All of the earlier comments apply, \textbf{except} that \textit{no} edges are used more than $O(1)$ times! This occurs because we have managed to ``clone" $\ell_n, r_n$ into effectively $(n-1)$ vertices, resulting in effectively $(n-1)^{2}$ edges between them; this exactly matches the $\Theta(n^{2})$ congestion of the original discrete path. Thus, applying  the formula in \cref{finitecomparison} naively would give a relaxation time that is $O(1)$, which is not correct.

As discussed at the start of the section, the problem is that the formula of \cref{finitecomparison} counts how many paths \textit{exactly} reuse an edge, while the current continuous example merely \textit{compresses} edges without \textit{exactly reusing} them. Earlier work \citet{yuen2000applications,yuen2002generalization} adjusted for this compression by associating a Jacobian to its paths. Unfortunately, this approach seems to be quite restrictive. The obvious problem is that this approach requires paths to be sufficiently smooth. A less-obvious problem is that the method often gives very poor bounds. This happens even when (as in our barbell example) the continuous chain of interest has a nearly-identical continuous analogue for which comparison inequalities give nearly-sharp answers; see  the original papers \cite{yuen2000applications,yuen2002generalization} for examples of this phenomenon.

In section \ref{SecBarbellEx2}, we analyze $Q_{n}$ using our new theorem (and the same path) and obtain the correct answer.

\subsection{Hyperfinite Approximation of Congestion Ratio} \label{SecHypCon}

In this section, we will consider Markov chains on a state space of the form $X = \sqcup_{i =1}^{I} X_{i}$, where $X_{i}$ is of the form $[0,1]^{k_{i}}$ and the notation ``$\sqcup$" is used to emphasize the fact that these are disjoint copies of what may be the same space; we use the word ``component" to refer to such a copy, so that e.g. $X_{1}$ is one component of $X$. We now define a metric $d$ on $X$ as following: 
\begin{enumerate}
\item For every $i \in \{1,2,\ldots,I\}$ and every $x,y\in X_i$, define $d(x, y)$ to be the standard Euclidean distance between $x$ and $y$. 

\item For $x\in X_{i_1}$ and $y\in X_{i_2}$ where $i_1\neq i_2$, define $d(x, y)= 1 + \max_{i\in I}\sqrt{k_{i}}$. 

\end{enumerate}
It is easy to verify that $d$ is a metric on $X$ with the following property: 
\[ 
\inf_{i \neq j \in I} \inf_{x \in X_{i}, y \in X_{j}} d(x,y) > \sup_{i \in I} \sup_{x,y \in X_{i}} d(x,y).
\] 
In other words, every point in one ``part" $X_{i}$ is closer to \textit{every} other point in the same part than it is to \textit{any} point in \textit{any} other part $X_{j}$. We will always use the completion of the Borel $\sigma$-algebra on $X$ associated with $d$.

We start by constructing a specific hyperfinite representation $S$ of $X$. Pick an infinite $K\in \NSE{\Nats}$ and let $\delta t=\frac{1}{K!}$. Define $S=\{0,\delta t, 2\delta t,\dotsc, 1-\delta t\}$, $S_{i} = S^{d_{i}}$, and $S_X=\sqcup_{i=1}^{I} S_{i}$. For every $1 \leq i \leq I$ and $s\in S_i$, $B_{i}(s) \subset S_{i}$ is defined to be the small rectangle with length equals to $\delta t$ on each side and containing $s$ as its left lower point. Let $\lambda_{i}$ denote the Lebesgue measure on $X_{i}$. Then the hyperfinite representations $(S_i, \{B_{i}(s)\}_{s\in S_i})$ are uniform in the sense that $\NSE{\lambda_{i}}(B_{i}(s_1))=\NSE{\lambda_{i}}(B_{i}(s_2))$ for $s_1,s_2\in S_{i}$. 

We work with two transition kernels $\{g(x,1,\cdot)\}_{x\in X}$ and $\{\alt{g}(x,1,\cdot)\}_{x\in X}$ on $X$ with stationary distributions $\pi$ and $\alt{\pi}$, respectively. Denote by $\lambda$ the natural extension of the Lebesgue measure to $X$ defined by the following formula: for Lebesgue-measurable $A_{1} \subset X_{1},\ldots,A_{I} \subset X_{I}$, set 
\[ 
\lambda(A_{1}\cup\ldots\cup A_{I}) = \lambda_{1}(A_{1}) + \ldots + \lambda_{I}(A_{I}),
\] 
where $\lambda_{i}$ is the usual Lebesgue measure on $X_{i}$.  We assume that $\{g(x,1,\cdot)\}_{x\in X}$, $\{\alt{g}(x,1,\cdot)\}_{x\in X}$, $\pi$, $\alt{\pi}$ are dominated by $\lambda$. Let $k(\cdot), \alt{k}(\cdot)$ denote the density functions for $\pi, \alt{\pi}$, respectively. Let $f(x,\cdot), \alt{f}(x,\cdot)$ denote the density functions for $g(x,1,\cdot), \alt{g}(x,1,\cdot)$, respectively. We now extend notions of path and edge from finite Markov processes to general Markov processes. 

Let $E=\{(x,y): f(x,y)>0\}$. An $E$-path from $x$ to $y$ is a sequence $(e_1,e_2,\dotsc,e_n)$ of edges in $E$ such that $e_1=(x,x_1),e_2=(x_1,x_2),\dotsc, e_n=(x_{n-1},y)$ for some vertices $x_1,x_2,\dotsc,x_{n-1}\in X$. Let $Q(x,y)=k(x)f(x,y)$ for $x,y\in \Omega$. $\alt{E}$ and $\alt{Q}$ are defined similarly. 
For $(x,y)\in X^2$, we assume that there exists an $E$-path connecting $x$ and $y$. We fix such a $E$-path denote it by $\gamma_{xy}$. Let $\allpath=\{\gamma_{xy}: (x,y)\in X^2\}$ be the collection of all such paths. We are, of course, really interested in $\PathSpace=\{\gamma_{xy}: (x,y)\in \alt{E}\}$. Note that we can view $\gamma$ as a function from $X^2$ to $\allpath$. The nonstandard counterparts of $E, Q, \alt{E}, \alt{Q}, \allpath, \PathSpace$ are defined as nonstandard extensions of these objects. By the transfer principle, $\NSE{\gamma}$ is a function from $\NSE{X^2}$ to $\NSE{\allpath}$. To associate a nonstandard edge (path) with a standard edge (path), we need to define the distance between edges and paths.  Abusing notation slightly, we write:

\begin{definition}\label{defedgedis}
Let $e_1=(x_1,y_1)$ and $e_2=(x_2,y_2)$ be two edges in $E$ or $\alt{E}$. 
The distance between $e_1,e_2$ is defined to be $\metric(e_1,e_2)=\max\{d(x_{1}, x_{2}), \, d(y_{1}, y_{2})\}$. 

Let $\gamma_{a_{1}b_{1}}$ and $\gamma_{a_{2}b_{2}}$ be two paths in $\PathSpace$. 
The distance between $\gamma_{a_{1}b_{1}}$ and $\gamma_{a_{2}b_{2}}$ is defined to be 
$\metric(\gamma_{a_{1}b_{1}},\gamma_{a_{2}b_{2}})=\max\{d(a_{1},a_{2}), \, d(b_{1},b_{2})\}$ - that is, on paths  $W$ is a pseudometric on paths which only pays attention to the endpoints.
\end{definition}

The distance between nonstandard edges and paths can be defined similarly. 
Two nonstandard edges (paths) are infinitely close to each other if and only if both their start and end points are infinitely close to each other. We write $e_1\approx e_2$ (respectively $\gamma_1\approx \gamma_2$) if $\NSE{\metric}(e_1,e_2)\approx 0$ ( respectively $\NSE{\metric}(\gamma_{1},\gamma_{2})\approx 0$) for two nonstandard edges (paths) $e_1, e_2$ ($\gamma_{1},\gamma_{2}$).

We shall use $\{G_{i}(\cdot)\}_{i\in S_X}$, $\{\alt{G}_{i}(\cdot)\}_{i\in S_X}$, $\{H_{i}(\cdot)\}_{i\in S_X}$, $\{\alt{H}_{i}(\cdot)\}_{i\in S_X}$, $\Pi(\cdot)$ and $\alt{\Pi}(\cdot)$ to denote various corresponding hyperfinite objects defined in \cref{hyprocess}. We now define the hyperfinite counterparts of $E, Q, \alt{E}, \alt{Q}, \PathSpace$. Let $\hyper{E}=\{(i,j): H_{i}(\{j\})>0\}$. Note that $\hyper{E}$ is a subset of $S_X\times S_X$. An $\hyper{E}$-path from $i$ to $j$ is a sequence $(e_1,e_2,\dotsc,e_n)$ of edges in $\hyper{E}$ such that $e_1=(i,i_1),e_2=(i_1,i_2),\dotsc, e_n=(i_{n-1},j)$ for some vertices $i_1,i_2,\dotsc,i_{n-1}\in S_X$. Let $\hyper{Q}(i,j)=\Pi(\{i\})H_{i}(\{j\})$ for $i,j\in \Omega$. $\alt{\hyper{E}}$ and $\alt{\hyper{Q}}$ are defined similarly. 
We defer the definition of the collection of hyperfinite paths (hyperfinite counterpart of $\hyper{E}$) to later part of the section.
We now introduce the following assumption. To be precise, let $\proj_{i}$ denote the projection from $X^2$ to the $i$-th coordinate for $i\in \{1,2\}$.  

\begin{assumption}\label{assumptiondp}
There exist open sets $\biggoodset\subset X^2$ such that,
the density functions $f(x,y)$, $k(x)$ are bounded away from $0$ for all $(x,y)\in \biggoodset$ and all $x\in \proj_{1}(\biggoodset)$. 
\end{assumption} 

In \cref{sectiondfmarkov}, \cref{stradon} implies that division between two hyperfinite probabilities can be used to approximate nonstandard extensions of the density functions. We assume that $\{g(x,1,\cdot)\}_{x\in X}, \{\alt{g}(x,1,\cdot)\}_{x\in X}, \pi$ and $\alt{\pi}$ are continuous on $\biggoodset$.

\begin{assumption}\label{assumptiondcm}
Let $\biggoodset$ be the same open set as in \cref{assumptiondp}. 
The density function $f(\cdot,\cdot)$ is jointly continuous on $\biggoodset$ and $k(\cdot)$ is continuous on $\proj_{1}(\biggoodset)$.
The density function $\alt{f}(\cdot,\cdot)$ is jointly continuous on $X^2$ and $\alt{k}(\cdot)$ is continuous on $X$.  
\end{assumption}

\cref{assumptiondp} and \cref{assumptiondcm} imply the following result. 

\begin{lemma} \label{hypradonq}
Suppose \cref{assumptiondp} and \cref{assumptiondcm} hold for $\biggoodset$. 
Then, for $i_1,i_2,j_1,j_2\in S_X$ such that $(i_2,j_2)\in \ST^{-1}(\biggoodset)$, we have
\[
\frac{\alt{\Pi}(\{i_1\})\alt{H}_{i_{1}}(\{j_{1}\})}{\Pi(\{i_2\})H_{i_2}(\{j_2\})}
\approx \frac{\NSE{\alt{k}}(i_{1})\NSE{\alt{f}}(i_{1},j_{1})}{\NSE{k}(i_{2})\NSE{f}(i_{2},j_{2})}
\approx \frac{\alt{\Pi}(\{i_1\})\alt{G}_{i_{1}}(\{j_{1}\})}{\Pi(\{i_2\})G_{i_2}(\{j_2\})}.
\]
\end{lemma}
\begin{proof}
As $(i_2,j_2)\in \ST^{-1}(\biggoodset)$, we have $(i,j)\in \NSE{\biggoodset}$ and $i\in \NSE{\proj_{1}}(\NSE{\biggoodset})$
for every $(i,j)$ with $(i,j)\approx (i_2,j_2)$.  
Thus, by \cref{assumptiondp} and \cref{markovdensity}, 
we know that these fractions are well-defined.
Moreover, by \cref{markovdensity} again, we have 
\[
\frac{\alt{\Pi}(\{i_1\})\alt{H}_{i_{1}}(\{j_{1}\})}{\Pi(\{i_2\})H_{i_2}(\{j_2\})}
&=\frac{\frac{\alt{\Pi}(\{i_1\})\alt{H}_{i_{1}}(\{j_{1}\})}{\NSE{\lambda}(B(i_1))
\NSE{\lambda}(B(j_1))}}{\frac{\Pi(\{i_2\})H_{i_2}(\{j_2\})}{\NSE{\lambda}(B(i_2))\NSE{\lambda}(B(j_2))}}\\
&\approx \frac{\NSE{\alt{k}}(i_{1})\NSE{\alt{f}}(i_{1},j_{1})}{\NSE{k}(i_{2})\NSE{f}(i_{2},j_{2})}\\
&\approx \frac{\alt{\Pi}(\{i_1\})\alt{G}_{i_{1}}(\{j_{1}\})}{\Pi(\{i_2\})G_{i_2}(\{j_2\})}.
\]
\end{proof}

We also have the following result from \cref{assumptiondp} and \cref{assumptiondcm}. 

\begin{lemma}\label{hyperedge}
Suppose \cref{assumptiondp} holds for $\biggoodset$. 
Then $(i,j)$ is an element of $\NSE{E}$ and $\hyper{E}$ for every $i,j\in S_X$ such that $(i,j)\in \NSE{\biggoodset}$. 
\end{lemma}
\begin{proof}
Suppose $(i,j)\in \NSE{\biggoodset}$. 
By the transfer of \cref{assumptiondp}, $\NSE{f}(i,j)>0$ hence $(i,j)\in \NSE{E}$.
By the construction of the hyperfinite representation, there exist $A, B\in \NSE{\BorelSets X}$ such that 
\begin{enumerate}
\item $A\subset B(i)$ and $i\in A$
\item $B\subset B(j)$ and $j\in B$
\item $A\times B\subset \NSE{\biggoodset}$
\item $\NSE{\lambda}(A)>0$ and $\NSE{\lambda}(B)>0$.
\end{enumerate} 
Then, by the construction of $H$ (see \cref{closereverse}), we conclude that $H_{i}(\{j\})>0$ hence $(i,j)\in \hyper{E}$.
\end{proof}

As we mentioned at the beginning of the section, we fix an $E$-path for every $(x,y)\in X^2$. Let $\Diagonal = \{ (x,x) \, : \, x \in X\} $ denote the collection of diagonal points of $X^2$. Our choices of paths have been arbitrary so far. We impose the following assumption which ensures that our choices of paths consist of certain edges.

\begin{assumption}\label{assumptiongp}
Let $\biggoodset$ be the same open set as in \cref{assumptiondp} and let $\epsilon_0$ be a positive real number. Let 
\[
\alt{E}_{\epsilon_0}=\{(x,y)\in X^2\setminus \Diagonal: \inf_{(a,b)\in \alt{E}}\metric((x,y),(a,b))\leq \epsilon_0\}.
\]
There exists a $\goodset\subset \biggoodset$ such that $\goodset$ is closed under the subspace topology of $X^2\setminus \Diagonal$ and, for every $(x,y)\in \alt{E}_{\epsilon_0}$, there exists an $E$-path $\gamma_{xy}$ consisting of edges from $\goodset$.
\end{assumption}

Note that $\ST^{-1}(\goodset)\subset \biggoodset$ since $\goodset$ is a closed subset of $\biggoodset$.  
With \cref{assumptiongp} holds, for $(x, y)\in \alt{E}_{\epsilon_0}$, we let the path connecting $x$ and $y$ consists of edges from $\goodset$.
In order to generalize \cref{finitecomparison} to general Markov processes, we need to make sure that every edge is not contained in too many paths. We impose the following \emph{$\epsilon$-separated property}.

\begin{assumption}\label{assumptionsp}
Fix $\epsilon>0$. 
For every edge $e\in E\setminus \Diagonal$, let 
\[
S_{e}=\{\gamma_{xy}\in \allpath: e\in \gamma_{xy}\}
\]
be the collection of all paths that contains the edge $e$. Then for any pair $(x_i,y_i), (x_j,y_j)$ such that $\gamma_{x_{i}y_{i}}, \gamma_{x_{j}y_{j}}\in S_{e}$, the following statement holds: 
\[
((x_i=x_j)\vee (d(x_{i},x_{j})>\epsilon))\wedge ((y_i=y_j)\vee  (d(y_{i},y_{j})>\epsilon)). 
\]
\end{assumption}

Fix some $\epsilon>0$. \cref{assumptionsp} then implies that, for each standard edge $e\not\in \Diagonal$, there is only finitely many paths from $\Theta$ contain it. It is also true that every nonstandard edge $v\not\in \NSE{\Diagonal}$ is contained in finitely many elements of $\NSE{\allpath}$. We assume that there is an uniform bound on the length of all paths. 
\begin{assumption}\label{assumptionub}
There exists $R>0$ such that $|\gamma|\leq R$ for every $\gamma\in \allpath$. 
\end{assumption}
We now impose the following regularity condition on the length of the paths in $\PathSpace$.  

\begin{assumption}\label{assumptionpl}
There exists $M>0$ with the following property: for every $\gamma\in \PathSpace$, there exists $\delta>0$ such that, for every $\gamma'\in \PathSpace$ with $d(\gamma,\gamma')<\delta$, we have $\bigl||\gamma|-|\gamma'|\bigr|\leq M$. 
\end{assumption}

\cref{assumptionpl} implies the following result: 

\begin{lemma}\label{lengthclose}
Suppose \cref{assumptionpl} holds. Let $\NSE{\gamma}_{x_{1}y_{1}}$ and $\NSE{\gamma}_{x_{2}y_{2}}$ be two nonstandard paths in $\NSE{\PathSpace}$. If $x_1\approx x_2$ and $y_1\approx y_2$ (that is, $\NSE{\metric}(\NSE{\gamma}_{x_{1}y_{1}}, \NSE{\gamma}_{x_{2}y_{2}})\approx 0$), then $\bigl||\NSE{\gamma}_{x_{1}y_{1}}|-|\NSE{\gamma}_{x_{2}y_{2}}|\bigr|\leq M$. 
\end{lemma}

We now define the collection of hyperfinite paths which we shall be primarily working with for the rest of the section. 
For $(i,i)\in \alt{\hyper{E}}\cap \NSE{\Diagonal}$, the hyperfinite path $\Gamma_{ii}$ is simply $(i,i)$ which has length $0$.
For $(i,j)\in \alt{\hyper{E}}\setminus \NSE{\Diagonal}$, by the transfer principle, 
there exists a nonstandard path $\NSE{\gamma}_{ij}\in \NSE{\allpath}$ connecting $i$ and $j$.  
By \cref{assumptionub}, we know that $|\NSE{\gamma}_{ij}|=k$ for some natural number $k\leq R$. Thus, $\NSE{\gamma}_{ij}$ is a sequence of edges in $\NSE{E}$ and we can denote this sequence by $(i, i_{1}), (i_{1},i_{2}), \dotsc, (i_{k-1},j)$. 
Note that, as $\alt{H}_{i}(\{j\})>0$, by the construction of $\alt{H}$ and the fact that $i\neq j$, we know that $(i,j)\in \NSE{\alt{E}_{\epsilon_0}}$. 
A hyperfinite path $\Gamma_{ij}$ from $i$ to $j$ is defined by the sequence 
\[
(i, s_{i_{1}}), (s_{i_{1}},s_{i_{2}}),\dotsc, (s_{i_{k-1}},j) 
\]
where, for every $1\leq n\leq k-1$, $s_{i_{n}}$ is the unique element in $S_X$ such that $i_{n}\in B(s_{i_{n}})$. 

We now verify that $\Gamma_{ij}$ is a well-defined hyperfinite path from $i$ to $j$. 
\begin{lemma}\label{hypathdefined}
Suppose \cref{assumptiondp} holds for $\biggoodset$, \cref{assumptiongp} holds for $\goodset$. 
Let $(a,b)\in \{(i,s_{i^{(1)}}),(s_{i^{(1)}},s_{i^{(2)}}),\dotsc, (s_{i^{(k-1)}},j)\}$.
Then, $H(a, b)>0$. 
\end{lemma}
\begin{proof}
Let $(a,b)\in \{(i,s_{i_{1}}),(s_{i_{1}},s_{i_{2}}),\dotsc, (s_{i_{k-1}},j)\}$. By \cref{assumptiongp} and the fact that $(i,j)\in \NSE{\alt{E}_{\epsilon_0}}$, we know that $(a,b)\in \NSE{\goodset}\subset \NSE{\biggoodset}$. By \cref{hyperedge}, we know that $H_{a}(\{b\})>0$.
\end{proof}
Thus, $(a,b)$ is an element of $\hyper{E}$ so $\Gamma_{ij}$ is a well-defined hyperfinite path. Let $\HPathSpace=\{\Gamma_{ij}: (i,j)\in \alt{\hyper{E}}\}$ be a fixed collection of hyperfinite paths. By construction, we have $|\Gamma_{ij}|\leq |\NSE{\gamma}_{ij}|$ for $(i,j)\in S_X\times S_X$. 

Suppose \cref{assumptionsp} ($\epsilon$-separated property) holds for some $\epsilon>0$. Let $v=(s,t)\in \NSE{E}\setminus \NSE{\Diagonal}$ for $s,t\in S_X$. By \cref{assumptionsp}, $v$ is contained in finitely many nonstandard paths (elements of $\NSE{\allpath}$). However, $v$ may be contained in infinitely many hyperfinite paths (elements of $\HPathSpace$). This is because, for a nonstandard path $\NSE{\gamma}_{ij}$ with $i,j\in S_X$, if $\NSE{\gamma}_{ij}$ contains a nonstandard edge $e=(x,y)$ where $x\in B(s)$ and $y\in B(t)$, then the hyperfinite path $\Gamma_{ij}$ contains $v$. So we impose the following regularity condition on $\allpath$.

\begin{assumption}\label{assumptionrp}
There exist $K\in \PosReals$ and a continuous function $m: E\setminus \Diagonal\to [0,K]$ with the following property: 
for every $\epsilon>0$ and standard edges $e,e'\in E\setminus \Diagonal$,  if $\metric(e,e')<\epsilon$ and $e\in \gamma\in \allpath$, there exists $\gamma'\in \allpath$ such that $\metric(\gamma,\gamma') <m(e)\epsilon$ and $e'\in \gamma'$. 
\end{assumption}

By the transfer principle, $\NSE{m}$ is an internal function from $\NSE{E}\setminus \NSE{\Diagonal}$ to $\NSE{\Reals}$. 
\cref{assumptionrp} also implies that $\NSE{m}(e)\approx \NSE{m}(e')$ for $e,e'\in \NSE{E}\setminus\NSE{\Diagonal}$ such that $\NSE{\metric}(e,e')\approx 0$ and $\ST(e)\not\in \Diagonal$.  
We have the following result. 

\begin{lemma}\label{hypathfinite}
Suppose \cref{assumptionrp} holds. 
Let $v=(i,j)\in \NSE{E}\setminus \NSE{\Diagonal}$ be a hyperfinite edge that is contained in $n$ nonstandard paths (elements of $\NSE{\Theta}$). 
Then $v$ is contained in at most $2Kn$ many hyperfinite paths. 
Suppose that $\ST(v)=(\ST(i),\ST(j))\in E\setminus \Diagonal$.
Then $v$ is contained in at most $2(\ceiling{m(\ST(v))}+1)\times n$ many hyperfinite paths. 
\end{lemma}
\begin{proof} 
Recall that each element in $\{B(s): s\in S_X\}$ is a rectangle with side length $\delta t$. 
Let $A=\{\NSE{\gamma}_{k}: k\leq n\}$ denote the collection of nonstandard paths that contain $v$ 
and let $\Gamma$ be a hyperfinite path that contains $v$. 
By the construction of hyperfinite path, there exists a nonstandard path $\NSE{\gamma}'$ that corresponds with $\Gamma$. 
Note that $\NSE{\gamma}'$ has the same start and end points as $\Gamma$ and $\NSE{\gamma}'$ contains some nonstandard edge $e=(x,y)\in \NSE{E}\setminus \NSE{\Diagonal}$ where $x\in B(i)$ and $y\in B(j)$. 
As $\NSE{\metric}(e,v)<\delta t$, by \cref{assumptionrp}, there exists $k_0\leq n$ such that $\NSE{\metric}(\NSE{\gamma}', \NSE{\gamma}_{k_0})<\NSE{m}(v)\delta t$. 
By \cref{assumptionrp} and the hypothesis of the theorem, we know that $v$ is contained in at most $2Kn$ many hyperfinite paths. 
If $\ST(v)=(\ST(i),\ST(j))\in E\setminus \Diagonal$, then $v$ is contained in at most $2(\ceiling{m(\ST(v))}+1)\times n$ many hyperfinite paths. 
\end{proof} 

\begin{lemma}\label{stpathcm}
Suppose \cref{assumptiondp} holds for $\biggoodset$.
Suppose \cref{assumptionsp} holds for some $\epsilon>0$ and \cref{assumptionrp} holds.
Let $\goodset$ be a subset of $\biggoodset$ such that $\ST^{-1}(\goodset)\subset \NSE{\biggoodset}$. 
Let $v=(i,j)\in (S_X\times S_X)\cap \ST^{-1}(\goodset)$ be contained in $\Gamma_{st}\in \HPathSpace$ with $i\not\approx j$.
Then $e=(\ST(i),\ST(j))$ is an element of $E\setminus \Diagonal$ and it is contained in $\gamma_{\ST(s)\ST(t)}\in \allpath$.  
\end{lemma}
\begin{proof} 
By \cref{hyperedge}, we have $v\in \NSE{E}\setminus \NSE{\Diagonal}$.
Note that $e\in \goodset\subset \biggoodset$ and $e\not\in \Diagonal$. 
By \cref{assumptiondp}, we know that $f(\ST(i),\ST(j))>0$ so $e\in E\setminus \Diagonal$. 
Let $\NSE{\gamma}_{st}$ denote the nonstandard path that associates with $\Gamma_{st}$. 
Then $\NSE{\gamma}_{st}$ contains $v'=(x,y)$ where $x\in B(i)$ and $y\in B(j)$. 
As $\NSE{\metric}(v', \NSE{e})\approx 0$, by \cref{assumptionrp}, we have $\NSE{m}(v')\approx \NSE{m}(\NSE{e})=m(e)$ and there exists a nonstandard path $\NSE{\gamma}'$ such that $\NSE{\metric}(\NSE{\gamma}_{st},\NSE{\gamma}')\approx 0$ and $\NSE{e}\in \NSE{\gamma}'$. 
By \cref{assumptionsp}, $e$ is only contained in finitely many paths. 
By the transfer principle, it must be contained in $\gamma_{\ST(s)\ST(t)}$. 
\end{proof} 

Before we present our main result, we introduce one more assumption which asserts that edges with small length can only be contained in paths with close starting and ending points. For $e=(a, b)\in E$, the length of $e$, which we denote by $\|e\|$, is defined to be $ \| e \| = d(a,b)$. 

\begin{assumption}\label{assumptionnstpa}
There exists a function $L: \Reals \to \Reals$ such that 
\begin{itemize}
\item $L(x)>0$ for all $x>0$,
\item $\NSE{L}(x)\approx 0$ if and only if $x\approx 0$,
\item For every $\gamma_{ab}\in \allpath$ and every $e\in \gamma_{ab}$, we have $\|e\|\geq L(d(a,b))$.
\end{itemize}
\end{assumption}

Note that, for every continuous function $L$ such that $L(0)=0$ and $L(x)>0$ for $x>0$, we know that $\NSE{L}(x)\approx 0$ if and only if $x\approx 0$.

\begin{lemma}\label{nshortpath}
Suppose \cref{assumptionnstpa} holds. 
Let $v=(s,t)$ be an element of $\hyper{E}$ such that $s\approx t$.
If $\Gamma_{ij}\in \HPathSpace$ contains $v$, then $i\approx j$. 
\end{lemma}
\begin{proof}
Let $\Gamma_{ij}\in \HPathSpace$ contain $v$. 
Then the nonstandard path $\NSE{\gamma}_{ij}$ contains $e=(x, y)\in \NSE{E}$ for $x\in B(s)$ and $y\in B(t)$. 
By \cref{assumptionnstpa} and the transfer principle, we have $\|e\|\geq \NSE{L}(|i-j|)$.
As $\|e\|\approx 0$, we know that $\NSE{L}(|i-j|)\approx 0$. 
By \cref{assumptionnstpa} again, we can conclude that $i\approx j$. 
\end{proof}

We are now at the place to establish the main result of this section.

\begin{theorem}\label{mainthm}
Let $\{g(x,1,\cdot)\}_{x\in X}$ and $\{\alt{g}(x,1,\cdot)\}_{x\in X}$ be two reversible transition kernels satisfying \cref{assumptionwsf} with stationary distributions $\pi$ and $\alt{\pi}$, respectively. 
Let $k(\cdot), \alt{k}(\cdot)$ denote the density functions for $\pi, \alt{\pi}$, respectively and let $f(x,\cdot), \alt{f}(x,\cdot)$ denote the density functions for $g(x,1,\cdot), \alt{g}(x,1,\cdot)$, respectively.
Suppose \cref{assumptiondp}, \cref{assumptiondcm} hold for some $\biggoodset$,  \cref{assumptiongp} holds for some $\goodset$ and $\epsilon_0>0$,  \cref{assumptionsp} holds for some $\epsilon>0$, \cref{assumptionub} holds for some $R>0$, \cref{assumptionpl} holds for some $M>0$, \cref{assumptionrp} holds for some continuous function $m: E\to \Reals$, and \cref{assumptionnstpa} holds for some $L: \Reals\to \Reals$. Let
\[\label{defepath}
B=\sup_{e\in E} \left( \frac{2(\ceiling{m(e)}+1)}{Q(e)}\sum_{(x,y): e\in \gamma_{xy}}\alt{Q}(x,y)(|\gamma_{xy}|+M) \right). 
\] 
Then, for every continuous function $h: X\to \Reals$, we have
\[
\Diric{\alt{g}}{h}{h}\leq B\Diric{g}{h}{h}.
\]
\end{theorem}
\begin{proof}
Let $\{H_{i}(\cdot)\}_{i\in S_X}$ and $\{\alt{H}_{i}(\cdot)\}_{i\in S_X}$ be two hyperfinite transition kernels associated with $\{g(x,1,\cdot)\}_{x\in X}$ and $\{\alt{g}(x,1,\cdot)\}_{x\in X}$ as defined in \cref{closereverse}. 
Let $\Pi(\cdot)$ and $\alt{\Pi}(\cdot)$ be *stationary distributions for $\{H_{i}(\cdot)\}_{i\in S_X}$ and $\{\alt{H}_{i}(\cdot)\}_{i\in S_X}$, respectively. 
By \cref{closereverse}, we know that $\{H_{i}(\cdot)\}_{i\in S_X}$ is *reversible with respect to $\Pi$ and $\{\alt{H}_{i}(\cdot)\}_{i\in S_X}$ is *reversible with respect to $\alt{\Pi}$. 
Let
\[
\hycongest=\max_{v\in \hyper{E}} \left(\frac{1}{\hyper{Q}(v)}\sum_{(s,t): v\in \Gamma_{st}}\alt{\hyper{Q}}(s,t)|\Gamma_{st}|\right)
\]
be the hyperfinite congestion ratio. 
By the transfer principle, for every internal function $F: S_X\to \NSE{\Reals}$, we have 
\[
\hyDiric{\alt{H}}{F}{F}\leq \hycongest \hyDiric{H}{F}{F}. 
\]

Pick a continuous function $h_0: X\to \Reals$. Define an internal function $F_0: S_X\to \NSE{\Reals}$ by letting $F_0(s)=\NSE{h_0}(s)$ for every $s\in S_X$. By \cref{hydstd}, we have $\Diric{g}{h_0}{h_0}\approx \hyDiric{H}{F_0}{F_0}$ and $\Diric{\alt{g}}{h_0}{h_0}\approx \hyDiric{\alt{H}}{F_0}{F_0}$. Thus, to finish the proof, it is sufficient to show that $\hycongest\lessapprox B$. 

Pick $v=(i,j)\in \hyper{E}$. 
Suppose $(B(i)\times B(j))\cap \NSE{\goodset}=\emptyset$. 
Let $\Gamma_{ab}\in \HPathSpace$.  
As $\NSE{\alt{H}}_{a}(\{b\})>0$, we conclude that $(a,b)\in \NSE{\alt{E}_{\epsilon_0}}$. 
By \cref{assumptiongp} and the transfer principle, 
we know that the nonstandard path $\NSE{\gamma}_{ab}$ is formed by edges from $\NSE{\goodset}$. 
As $(B(i)\times B(j))\cap \NSE{C}=\emptyset$, we know that $v\not\in \Gamma_{ab}$.
Thus, we can conclude that 
\[
\frac{1}{\hyper{Q}(v)}\sum_{(s,t): v\in \Gamma_{st}}\alt{\hyper{Q}}(s,t)|\Gamma_{st}|=0.
\]

We now suppose that $(B(i)\times B(j))\cap \NSE{C}\neq \emptyset$ but but $i\approx j$. 
By \cref{assumptionsp}, $v$ is contained in finitely many nonstandard paths. 
By \cref{hypathfinite}, $v$ is contained in finitely many hyperfinite paths. 
By \cref{nshortpath}, any hyperfinite path contains $v$ has infinitesimal length. 
Hence, we have
\[
\frac{1}{\hyper{Q}(v)}\sum_{(s,t): v\in \Gamma_{st}}\alt{\hyper{Q}}(s,t)|\Gamma_{st}|=0
\]
in this case. 
 
We now suppose that $v\in \ST^{-1}(\goodset)\subset \ST^{-1}(\biggoodset)\subset \NSE{\biggoodset}$. 
This immediately implies that $i\not\approx j$. 
By \cref{assumptiondp}, we know that $\ST(\NSE{f}(i,j))>0$ and $\ST(\NSE{k}(i))>0$. 
Let $\Gamma_{ab}$ be a hyperfinite path that contains $v$. 
Suppose $\NSE{\alt{f}}(a,b)\approx 0$. 
Then, by \cref{hypradonq}, we know that $\frac{\alt{\hyper{Q}}(a,b)}{\hyper{Q}(v)}\approx 0$. 
By \cref{assumptionsp} and \cref{hypathfinite}, $v$ is contained in finitely many hyperfinite paths. 
Thus, we have 
\[\label{eqrestrict}
\frac{1}{\hyper{Q}(v)}\sum_{(s,t): v\in \Gamma_{st}}\alt{\hyper{Q}}(s,t)|\Gamma_{st}|\approx \frac{1}{\hyper{Q}(v)}\sum_{v\in \Gamma_{st}\in \mathcal{T}_{v}}\alt{\hyper{Q}}(s,t)|\Gamma_{st}|
\]
where $\mathcal{T}_{v}=\{\Gamma_{st}: \ST(\NSE{\alt{f}}(s,t))>0\wedge v\in \Gamma_{st}\}$. 

By \cref{stpathcm}, $e=(\ST(i),\ST(j))$ is an element of $E\setminus \Diagonal$. It follows from \cref{assumptionsp} that $e$ is contained in finitely many elements of $\PathSpace$. Let 
\[
T_{e}=\{\gamma\in \PathSpace: e\in \gamma\}=\{\gamma_{x_{1}y_{1}},\gamma_{x_{2}y_{2}},\dotsc,\gamma_{x_{n}y_{n}}\}
\]
for some $n\in \Nats$.
By \cref{stpathcm} again, we have $\mathcal{T}_{v}=\bigcup_{i=1}^{n}A_i$ where $A_i=\mathcal{T}_{v}\cap \{\Gamma\in \HPathSpace: \NSE{\metric}(\Gamma, \NSE{\gamma}_{x_{i}y_{i}})\approx 0\}$. By \cref{hypathfinite}, for every $1\leq i\leq n$, we have $|A_i|\leq \ceiling{2m(e)+1}$. 
By \cref{hypradonq}, \cref{assumptiondp} and \cref{assumptiondcm} , we have
\[
\frac{1}{\hyper{Q}(v)}\sum_{v\in \Gamma_{st}\in \mathcal{T}_{v}}\alt{\hyper{Q}}(s,t)|\Gamma_{st}|\approx \sum_{v\in \Gamma_{st}\in \mathcal{T}_{v}}\frac{\alt{Q}(\ST(s),\ST(t))}{Q(e)}|\Gamma_{st}|.
\]
By \cref{lengthclose}, we have $|\Gamma_{st}|\leq |\gamma_{\ST(s)\ST(t)}|+M$. Thus, combining with \cref{eqrestrict}, we have
\[
\sum_{(s,t): v\in \Gamma_{st}}\frac{\alt{Q}(\ST(s),\ST(t))}{Q(e)}|\Gamma_{st}|\lessapprox 
\frac{2(\ceiling{m(e)}+1)}{Q(e)}\sum_{(x,y): e\in \gamma_{xy}}\alt{Q}(x,y)(|\gamma_{xy}|+M).
\]
Hence, we can conclude that $\hycongest\lessapprox B$, proving the desired result. 
\end{proof}

\subsection{Example: Finishing the Barbell Graph} \label{SecBarbellEx2}

We continue from Section \ref{SecBarbellEx1}, using the same notation and  paths. We will compare the kernel $Q_{n}'$ to the kernel that takes i.i.d. samples from the stationary measure of $Q_{n}'$. Let $d$ be as in Section \ref{SecHypCon}.


We now quickly check the assumptions of \cref{mainthm}: 

\begin{itemize}
\item Both transition kernels clearly satisfy \cref{assumptiondsf} (indeed, in the statement of the condition one can take $\delta = 0.1$ for \textit{all} $\epsilon \geq 0$), and thus they also satisfy the weaker condition \cref{assumptionwsf}.
\item \cref{assumptiondp} and \cref{assumptiondcm} hold with $\biggoodset = X^{2}$.
\item  \cref{assumptiongp} holds with $\goodset$ consisting of all edges of length at most 1, and $\epsilon_{0} = 1$. 
\item \cref{assumptionsp} holds with any $0 < \epsilon < 1$. 
\item \cref{assumptionub} holds with $R=3$.
\item  \cref{assumptionpl} holds trivially with $M=R-1=2$.
\item \cref{assumptionrp} holds for any $0 < \epsilon < 1$ with  $m(e) \equiv (n+1)$.
\item \cref{assumptionnstpa} holds with $L(x) \equiv x$.
\end{itemize}

Applying these bounds to the formula $B$ appearing in \cref{mainthm}, we find $B = \Theta(n^{2})$; this gives the correct dependence on $n$ for the relaxation time.

\printbibliography